\documentclass[a4paper,10pt]{amsart}
\usepackage[utf8]{inputenc}
\usepackage{amsthm}
\usepackage{amsmath}
\usepackage{bm}
\usepackage{amsfonts}
\usepackage{amssymb}
\usepackage{mathtools}
\usepackage{appendix}
\usepackage{mathrsfs}
\usepackage{setspace}
\usepackage{xcolor}
\usepackage{todonotes}
\usepackage{thmtools} 
\usepackage[foot]{amsaddr}
\usepackage[pdfdisplaydoctitle,colorlinks,breaklinks,urlcolor=blue,linkcolor=blue,citecolor=blue]{hyperref} 

\newcommand{\B}{\mathcal{B}}

\newcommand{\D}{\mathcal{D}}
\newcommand{\DD}{\mathbb{D}}

\newcommand{\F}{\mathcal{F}}

\newcommand{\HH}{\mathcal{H}}

\renewcommand{\L}{\mathscr{L}}

\newcommand{\N}{\mathbb{N}}

\newcommand{\PP}{\mathbb{P}}

\newcommand{\R}{\mathbb{R}}

\renewcommand{\S}{\mathscr{S}}

\newcommand{\Z}{\mathbb{Z}}

\let\div\relax
\DeclareMathOperator{\lip}{Lip}
\DeclareMathOperator{\div}{div}

\renewcommand{\epsilon}{\varepsilon}
\renewcommand{\setminus}{\smallsetminus}

\newcommand{\one}{\bm{1}}

\newcommand{\de}{\partial}

\newcommand{\set}[1]{\left\{#1\right\}}
\newcommand{\pa}[1]{\left(#1\right)}
\newcommand{\bra}[1]{\left[#1\right]}
\newcommand{\abs}[1]{\left|#1\right|}
\newcommand{\norm}[1]{\left\|#1\right\|}
\newcommand{\brak}[1]{\left\langle#1\right\rangle}

\newcommand{\prob}[1]{\mathbb{P}\left(#1\right)}
\newcommand{\expt}[2][]{\mathbb{E}_{#1}\left[#2\right]}

\newtheorem{theorem}{Theorem}[section]
\newtheorem{definition}[theorem]{Definition}

\newtheorem{corollary}[theorem]{Corollary}
\newtheorem{lemma}[theorem]{Lemma}
\newtheorem{proposition}[theorem]{Proposition}

\theoremstyle{remark}
\newtheorem{remark}[theorem]{Remark}

\numberwithin{equation}{section}

\newenvironment{acknowledgements}{%
  \begin{abstract}
}{%
  \end{abstract}
}

\title[Uniform Approximation of 2D Navier-Stokes]{Uniform Approximation of 2D Navier-Stokes Equations with Vorticity Creation by Stochastic Interacting Particle Systems}
\author[F. Grotto]{Francesco Grotto}
\address{Università di Pisa, Dipartimento di Matematica, 5 Largo Pontecorvo, 56127 Pisa, Italia}
\email{\href{mailto:francesco.grotto at unipi.it}{francesco.grotto at unipi.it}}
\author[E. Luongo]{Eliseo Luongo}
\address{Scuola Normale Superiore, Piazza dei Cavalieri, 7, 56126 Pisa, Italia}
\email{\href{mailto:eliseo.luongo at sns.it}{eliseo.luongo at sns.it}}
\author[M. Maurelli]{Mario Maurelli}
\address{Università di Pisa, Dipartimento di Matematica, 5 Largo Pontecorvo, 56127 Pisa, Italia}
\email{\href{mailto:mario.maurelli at unipi.it}{mario.maurelli at unipi.it}}
\date\today

\begin{document}

\begin{abstract}
	We consider a stochastic interacting particle system
	in a bounded domain with reflecting boundary,
    including creation of new particles on the boundary prescribed by a given source term.
    We show that such particle system approximates
	2d Navier-Stokes equations in vorticity form and impermeable boundary, 
    the creation of particles modeling vorticity creation at the boundary.
	Kernel smoothing, more specifically smoothing by means of
	the Neumann heat semigroup on the space domain, 
    allows to establish uniform convergence
	of regularized empirical measures
	to (weak solutions of) Navier-Stokes equations.
\end{abstract}

\maketitle

\section{Introduction}\label{sec:introduction}

Let $D\subset \R^2$ be a smooth, bounded, open 2-dimensional domain,
and consider 2d Navier-Stokes equation under no-slip boundary conditions,
\begin{equation}\label{eq:2dNS}
    \begin{cases}
        \de_t u+ (u\cdot \nabla) u + \nabla p= \nu \Delta u, &\text{in }[0,T]\times D,\\
        u=0, &\text{in }[0,T]\times \de D.
    \end{cases}
\end{equation}
Reformulating the bulk dynamics in terms of vorticity 
$\omega= \nabla^\perp\cdot u$ (the curl of $u$) straightforwardly leads to
\begin{equation*}
    \de_t \omega + u\cdot \nabla \omega= \nu \Delta \omega,
\end{equation*}
an equation that can be put in closed form
expressing $u=K[\omega]$ in terms of $\omega$
by means of Biot-Savart law,
\begin{equation*}
    K[\omega](x)=\int_D K(x,y)\omega(y)dy,
    \quad K(x,y)=\nabla_x^\perp G(x,y),\quad
    G=(-\Delta_{Dir})^{-1},
\end{equation*}
where the choice of Dirichlet conditions for the Green function
$(-\Delta_{Dir})^{-1}$
makes it so that the vorticity dynamics encodes the impermeability
condition $u|_{\de D}\cdot \hat n=0$ at the boundary.
The tangential part $u^\perp|_{\de D}\cdot \hat n=0$ 
of the no-slip boundary condition, on the other hand,
is harder to express in terms of $\omega$ only, giving rise
to a non-local condition on vorticity (see \cite[Section 1.1]{Bardos2022} for an explicit expression).
Physically speaking, the condition $u^\perp|_{\de D}\cdot \hat n=0$ 
is expected to force production of vorticity, especially in the proximity of the boundary,
so that the tangential velocity at $\de D$ induced by the bulk of the fluid be compensated \cite{Benfatto1984}.
This fact becomes an issue in the study of approximation methods
for Navier-Stokes equations in vorticity form, which is a
natural approach in the 2-dimensional setting.

Previous approaches to approximating vorticity creation at the boundary for 2-dimensional Navier-Stokes
equations \cite{Chorin1978,Benfatto1984,Cottet1988,Cottet1995,Jourdain2004,Meleard2004}
have considered (unphysical) Neumann boundary conditions for the vorticity dynamics,
    \begin{equation}\label{eq:neuNS}\tag{NNS}
        \begin{cases}
            \de_t \omega+K[\omega]\cdot \nabla \omega =\nu \Delta \omega,\quad \text{in }D\times [0,T],\\
            \nabla\omega\cdot \hat n=g, \quad \quad \text{in }\de D\times (0,T), 
        \end{cases}
    \end{equation}
where $g\in L^2([0,T], L^2(\de D))$ and $\nu>0$.
The latter system acts as a natural proxy for the (physical) no-slip condition if one seeks to mathematically describe vorticity production
in the boundary layer by means of the source term of Neumann boundary condition or with an explicit addition
of vorticity sheets, recovering the correct PDE dynamics \eqref{eq:2dNS} in a macroscopic or iterative limit.
The derivation of \eqref{eq:2dNS} on a generic domain $D$
by means of an approximation method explicitly describing the effect of the boundary as a singular source of vorticity
remains a mostly open problem.

The present contribution aims to lay the groundwork for the construction of an approximation method for \eqref{eq:2dNS}
based on interacting particle systems uniformly converging to the PDE system, that is, seeking to
provide at the same time a good microscopic description of vorticity creation, and a good approximation scheme.
In order to do so, we generalize to the bounded domain setting the analytic and probabilistic tools introduced in \cite{Flandoli2019,Flandoli2020}
for the uniform approximation of PDEs by (stochastic) interacting particle systems in full-space.

We consider a system of $N$ particles evolving according to the stochastic differential equations
with reflecting boundary
\begin{equation}\label{eq:particleSDE}
    d x_i = F\pa{\sum_{j=1}^N \omega_j K_n(x_i,x_j)}dt + \sqrt{2\nu} dB^i-dk_i,
\end{equation}
each particle starting its evolution at time $t_i\in [0,T]$ at $x_i(t_i)$, either located in the interior of $D$ if $t_i=0$ or on $\de D$ if $t_i>0$.
The system is closely related to the one considered in \cite{Jourdain2004},
the main difference being the generation mechanism.
Let us briefly and informally describe the various parts of the dynamics
(precise definitions are all deferred to \autoref{sec:statements}).

\emph{Generation of particles}, that is initialization of each of the SDEs \eqref{eq:particleSDE},
happens either in the interior $D$ at time $0$ 
(so to model the initial datum of the limiting PDE)
or at the boundary $\de D$ at later times 
(so to model the source term $g$ of the Neumann b.c.).
Concretely, we will consider a grid spanning 
$(\set{0}\times D) \cup ((0,T]\times \de D)$ indexed by $i$, 
and generate the particle $x_i$ at grid point $i$ with an intensity $\xi_i$
determined by a local average of the boundary datum around $x_i$.
The mesh of the grid (both in time and space) will be $1/n$, thus
$n\in \N$, $n\to\infty$ is a convenient parameter 
to rule the macroscopic limit:
for instance the global number of particles will have order $N=N(n)\simeq n^2$.
Weights (or intensities) $\omega_i$ of particles will correspond to local averages
of the initial datum $\omega_0$ or the boundary datum $g$ around the starting point of the trajectory of $x_i$.

\emph{Interaction of particles} is given by a regularized version of the singular kernel $K$,
obtained by applying the Neumann heat semigroup,
\begin{equation}
    K_n(x,y)=\int_D K(x,z)p_{\epsilon}(z,y)dz, \quad \epsilon=\epsilon(n), \quad x,y\in D,
\end{equation}
with $p_\epsilon$ the Neumann heat kernel on $D$ and $\epsilon\to 0$ with a
slow enough rate in terms of $n$.
Kernel smoothing by means of the heat semigroup $P_t$, compared for instance to smoothing by convolution with bump functions,
has many advantages in our setting: it is well suited for regularizing functions keeping their support in $D$,
$P_t$ preserves the $L^1$ norm of non-negative functions,
and precise estimates on $p_t$ and its derivatives are available.

The interaction term of \eqref{eq:particleSDE} also includes a cutoff,
\begin{equation}
    F:\R^2\to\R^2,\quad 
    F(v)=\frac{v}{|v|}(|v|\wedge M), \quad M>0,
\end{equation}
(which is Lipschitz continuous, and its Lipschitz constant is uniformly bounded in $M>0$).
In principle, the presence of the cutoff makes it so that the macroscopic behavior of the particle system 
is described by a different PDE, that is
\begin{equation}\label{eq:cneuNS}\tag{c-NNS}
        \begin{cases}
            \de_t \omega+\operatorname{div}(F(K[\omega])  \omega) =\nu \Delta \omega,\quad \text{in }D\times [0,T],\\
            \nabla\omega\cdot \hat n=g, \quad \quad \text{in }\de D\times (0,T), 
        \end{cases}
\end{equation}
coinciding with \eqref{eq:neuNS} only when the velocity field in the latter is bounded.
As revealed by suitable \emph{a priori} estimates, this is in fact the case, so the cutoff $F$
is just a technical device that will naturally disappear in the macroscopic limit taking $M>0$ large enough.
Regularization of interactions is important to ease uniform convergence of empirical measures and well-posedness of the SDE system.

\emph{Brownian noise} acting independently on each particle models viscosity,
and boundary terms $-dk_i$ are defined so that $\de D$ acts on particles as a reflecting boundary,
in the standard setting of \cite{Stroock1971,Lions1984}. In terms of boundary conditions
for the macroscopic limit, this would produce the null Neumann condition $\hat n\cdot \nabla\omega=0$,
so an additional effect must be included to model the source $g$.

Our main result concerns the uniform convergence (in space)  of kernel-smoothed 
empirical measures of the particle system under consideration, 
\begin{equation}\label{eq:empiricalmeasures}
    S^n(t)=\sum_{i\in A_t^n}\omega_i \delta_{x_i(t)}, 
    \quad \omega^n(t,x)=P_{\epsilon}S^n(t),
\end{equation}
where, $P_\epsilon$ stands for the action of Neumann heat semigroup on measures,
and $A_t^n$ collects the indices of particles generated up to time $t$.
The generation of new particles at the boundary produces discontinuities
in time at the level of $\omega_n(t,x)$, therefore our limit theorem is going to
be set in the space of $C(\bar D)$-valued \emph{càdlàg} functions,
$\DD([0,T],C(\bar D))$, endowed with the usual Skorohod topology.

Informally, our main result (\autoref{thm:main}) is the following:
\emph{the regularized empirical measures $\omega_n(t,x)$ converge as $n\to\infty$, almost surely in the topology of $\DD([0,T],C(\bar D))$, to the unique weak solution of \eqref{eq:neuNS}
with Lipschitz continuous initial datum $\omega_0$ 
and bounded Neumann boundary datum $g$.}
As detailed below, we actually obtain a stronger convergence in Sobolev norms.
We will in fact assume that $g\geq 0$ in the exposition of our argument,
and remove this further assumption in the later \autoref{sec:generalid},
see \autoref{thm:mainsign}.

The proof or \autoref{thm:main} consists in a compactness argument, 
therefore the technical core of this paper resides in uniform estimates on the particle system.
The main additional difficulty compared to the works \cite{Flandoli2019,Flandoli2020} inspiring our technique is of course the presence of the boundary,
which complicates many analytic operations. 
For instance, the absence of translation invariance in the models under consideration
prevents derivatives and heat semigroups from commuting, so one has to resort instead to
suitable gradient estimates on the heat kernel in order to exploit the
regularizing properties of heat semigroup.
Uniform convergence at time $0$ and at the boundary also requires a careful control due to particle generation,
but we are able to deal with rather natural assumptions on $\omega_0,g$ 
(as compared with the obtained convergence),
and to avoid further technical hypothesis on initial data such as \cite[Assumption 1.1.3]{Flandoli2019} (a restricting requirement on the relation between the scale of kernel smoothing and sample size).

Beside the improved notion of convergence to the PDE we are able to obtain, the main point of our work is the approach to particle approximation for \eqref{eq:neuNS}: our approach is almost completely functional analytic
and based on semigroup theory,
in contrast with results in the perhaps more classical framework of propagation of chaos such as \cite{Jourdain2004,Fournier2014}. The motivation for our approach is that our technique circumvents 
the need for fine controls on particle dynamics.
We regard this as a step forward in the scope of attempting to build
a similar particle approximation of the physical no-slip Navier-Stokes equations
\eqref{eq:2dNS}, since in that case particle creation has to be linked with
the velocity at the boundary induced by the bulk of the fluid, 
making single particle dynamics very hard to treat (as already commented
in \cite{Jourdain2004}).

The particle system under analysis can be thought of as a (regularized) 
stochastic version of the well-known point vortex system.
However, there are some relevant differences concerning in particular the effect
of the boundary on the vortex dynamics, see \autoref{rmk:pointvortex} below.
Let us only briefly discuss here some qualitative aspects for the sake
of a comparison with the related literature.
Due to the regularization of interaction at the microscopic level
our approximants share some features with the systems of \emph{moderately interacting}
particles (see \cite{FlandoliHuang2021,Flandoli2019} for a comparison), but our focus is on the approximation procedure rather than on the scaling of
$\epsilon=\epsilon(n)$.
If we were to omit the regularization of the interaction kernel and the cutoff $F$
(and neglecting for a moment the complications of boundary effects),
the scaling of weights $\omega_i$ would make our particle system \eqref{eq:particleSDE}
a mean-field rescaled version of the stochastic point vortex model.
While the mean field limit of deterministic point vortices system
towards 2d Euler's equation in the well-posedness class is a recent result \cite{Rosenzweig2022}, our (stochastic) result (towards 2d Navier-Stokes) is of course closer to the ideas of \cite{Marchioro1982,Meleard2000}.
Concerning the relation with propagation of chaos theory, besides the above references closely related to our arguments, we refer for completeness to the work of Jabin-Wang \cite{Jabin2018} 
on stochastic particle systems with general interactions (on full space),
including stochastic vortex dynamics, for which we also mention 
\cite{Coghi2016} discussing more general noise.
The study of propagation of chaos in stochastic particle systems
in bounded domain dates back at least to \cite{Sznitman1984},
we refer to \cite{Coghi2022} for a discussion and recent results.

\section{Functional Analytic Setup}\label{sec:analfun}

In what follows, $dx$ denotes the Lebesgue measure restricted to the domain $D$, and $d\sigma(x)$ the (1-dimensional) volume measure on the smooth boundary $\de D$. In order to lighten notation, we will simply write $L^p=L^p(D,dx)$, avoiding ambiguity by explicitly indicating the underlying space for any different $L^p(X)$.

We denote by $\B(X)$, $\B_b(X)$ respectivly the spaces real-valued Borel and bounded Borel functions on a topological space $X$.

Landau $O$'s and $o$'s have their usual meaning, subscripts indicating eventual dependence on parameters.
The symbol $C$ will denote a positive constant, possibly differing in any of its occurrence
even in the same formula, depending only on eventual subscripts.
The expression $A\simeq_{a,b} B$ indicates that $B$ is both an upper and lower bound for $A$
up to multiplicative constants depending only on possible subscripts $a,b$. 
Expressions $A\lesssim_{a,b} B$ or $A\gtrsim_{a,b} B$ indicate respectively 
an upper and lower bound in the same sense.

If $p\in [1,\infty]$ we denote by $p'$ its conjugate exponent, $1/p+1/p'=1$.

\subsection{Sobolev Spaces and Neumann Heat Semigroup}
For any $\alpha\in \R,\ p\geq 1$ the \emph{Bessel potential space} is defined by
\begin{equation*}
	H^{\alpha,p}(\R^d) =\set{u\in \S'(\R^d): 
	\norm{u}_H^{\alpha,p}(\R^d)=\norm{\F^{-1}((1+\abs{\cdot}^2)^{\alpha/2}\F u(\cdot))}_{L^p(\R^d)}<\infty}.
\end{equation*}
If $D\subseteq \R^d$ is a smooth bounded domain we set
\begin{equation*}
H^{\alpha,p}(D)=\set{u|_{D}: u\in H^{\alpha,p}(\R^d)}, \quad
\norm{u}_{H^{\alpha,p}(D)}:=\inf \set{\norm{w}_{H^{s,p}(\R^d)}:\ w|_{D}=u}.
\end{equation*}

The following statement summarizes standard results,
we refer to monographs \cite{grubb2016regularity,trie1983fun,trie1995fun} for a complete discussion.

\begin{proposition} It holds:
	\begin{itemize}
		\item (Sobolev embeddings) if $1<p\leq q<\infty$, and $\alpha\geq\beta\geq 0$,
		\begin{equation*}
			H^{\alpha,p}(D)\hookrightarrow H^{\beta,q}(D), \quad \alpha-\frac{d}{p}\geq \beta-\frac{d}{q};
		\end{equation*}
		embeddings are compact if the latter inequality is strict;
		\item (Morrey's inequality) if $\alpha-\frac{d}{p}>l+\lambda$, $l\in \N$ and $\lambda\in [0,1)$,
		\begin{equation*}
			H^{\alpha,p}(D)\stackrel{c}{\hookrightarrow}C^{l,\lambda}\left(\bar{D}\right).
		\end{equation*} 
	\end{itemize}
\end{proposition}

If $\alpha>1+1/p$, $p\in (1,\infty)$ and $u\in H^{\alpha,p}(D)$ then the trace of $\nabla u$ on $\de D$
(defined as usual as the extension of the restriction to $\de D$ for smooth functions) is of class $L^p(\de D)$.
Therefore, if $\hat n:\de D\to \R^2$ is the inward pointing unit normal vector field of $\de D$,
the following definition of Sobolev space with Neumann boundary conditions is meaningful:
\begin{equation*}
    H^{\alpha,p}_{Neu}\coloneqq \set{u\in H^{\alpha,p}(D),\quad \nabla u|_{\de D}\cdot \hat n\equiv 0}.
\end{equation*}
We refer to \cite[Sec. 4.3.3]{trie1995fun} for a 
more general discussion of boundary conditions for Sobolev spaces, and to \cite[Sec. 2.9.3]{trie1995fun}
for further regularity of the mentioned trace operator in terms of Besov spaces on $\de D$.

Consider the heat equation with Neumann boundary conditions
\begin{equation}\label{eq:neumannproblem}
\begin{cases}
\de_t u_t(x)= \nu \Delta u_t(x), & x\in D,\,t>0,\\
\hat{n}\cdot \nabla u_t(x)=0, &x\in \de D,\,t>0,\\
u_0(x)=f(x), &x\in D,
\end{cases}
\end{equation}
with $f$ a bounded measurable function on $D$. It is well-known that the PDE problem is well-posed,
we denote its solution by $P_tf(x)=u_t(x)$, $t\geq 0$, $x\in D$; moreover $P_tf\in C^\infty(D)$
for all $t>0$. Operators $P_t$, $t>0$ form a Markov semigroup (in the sense of \cite{Bakry2014}),
associated to the \emph{reflected Brownian motion} in $D$ (see \autoref{ssec:BMreflecting} below).

We now recall classical facts concerning the action of the semigroup $P_t$ on $L^p(D)$
and Sobolev spaces: we refer to \cite[Section 7.3]{Pazy1983} for a detailed discussion.
For all $p\in [1,\infty)$, $P_t$ extends to an analytic semigroup of contractions 
$P_t=e^{-tA_p}:L^p\to L^p$ with infinitesimal generator
\begin{equation*}
    -A_p:\D(A_p)=H^{2,p}_{Neu}\subset L^p(D)\rightarrow L^p(D),
\end{equation*}
coinciding with the Laplace operator $\nu \Delta$ on $C^\infty_c(D)$.
The spectrum of $A_p$ is included in $[0,+\infty)$. 
A first consequence of the fact that $P_t$ is an analytic semigroup
is \emph{ultracontractivity} on $L^p$ spaces:
we refer to \cite[Section 7.3.2]{Arendt2004} for the following result.

\begin{proposition}
	If $1\leq p<q\leq \infty$, then it holds
	\begin{equation}\label{eq:ultracontractivity}
		\norm{P_t}_{L^p\to L^q}\leq \frac{C_{T,\nu,p,q}}{t^{1/p-1/q}}, \quad t\in[0,T].
	\end{equation}
\end{proposition}

Another consequence of analiticity of $P_t$ is that we can 
introduce fractional powers of $I+A_p$ as operators on $L^p$ by means of 
\begin{equation*}
\left(I+A_p\right)^{-\alpha}=\frac1{\Gamma(\alpha)}\int_0^\infty t^{\alpha-1}e^{-t}P_tdt,
\quad \alpha>0,
\end{equation*}
where the integral converges in the uniform operator topology and defines injective operators
(\emph{cf.} \cite[Section 2.6]{Pazy1983}). 
We set
\begin{equation*}
	(I+A_p)^0=I, \quad (I+A_p)^\alpha=((I+A_p)^\alpha)^{-1} \text{ for }\alpha>0,
\end{equation*}
and denote by
\begin{equation*}
    \HH^{\alpha,p}= \D((I+A_p)^{\alpha/2})
\end{equation*}
the domain of fractional powers for all $\alpha\in\R$.

\begin{proposition}\label{prop:fractionalpowers} Let $p\in (1,\infty)$. It holds:
		\begin{itemize}
		\item (group property) for $\alpha,\beta\in\R$ and $u\in \HH^{2\gamma,p}$, $\gamma=\alpha\vee \beta\vee (\alpha+\beta)$,
		\begin{equation*}
			(I+A_p)^{\alpha+\beta}u=(I+A_p)^{\alpha}(I+A_p)^{\beta}u;
		\end{equation*}
		\item (adjoint) $A_p^*=A_{p'}$ with $\frac{1}{p}+\frac{1}{p'}=1$;
		\item (identifications of the domain\footnote{
		These identifications and Sobolev embeddings on $D$ are closely related to the ultracontractivity of $P_t$,
		we refer again to \cite{Arendt2004} for a discussion.
		}) $\HH^{2\alpha,p}$ can be identified with
		the interpolation space $[L^p, \D(A_p)]_\alpha$ for all $\alpha\in [0,1]$;
		moreover it holds
		\begin{equation*}
		\HH^{\alpha,p}=\begin{cases}
		H^{\alpha,p}(D) \quad \quad \quad \quad \quad \quad \quad \quad\quad\quad\  \textit{for}\quad 0\leq\alpha<1+\frac{1}{p};\\
		\{u\in H^{\alpha,p}(D):\hat{n}\cdot \nabla u|_{\partial D}=0  \} \quad \textit{for}\quad  1+\frac{1}{p}<\alpha<3+\frac{1}{p};\\
		\{u\in H^{\alpha,p}(D):\hat{n}\cdot \nabla A^l_pu|_{\partial D}=0\ \forall l\in\N,\ l\leq k  \}\\  \quad  \quad  \quad \quad \quad\quad\quad \quad\quad\quad\quad\quad\quad\quad\textit{for}\quad 1+2k+\frac{1}{p}<\alpha<3+2k+\frac{1}{p}.\\
		\end{cases}
		\end{equation*}
		\item (contractivity) for $\alpha\in\R$, $P_t(L^p)\subset \HH^{2\alpha,p}$ for all $t>0$
		and operators $P_t$ and $(I+A_p)^\alpha$ commute on $\HH^{2\alpha,p}$;
		moreover
		\begin{equation}\label{eq:contractivity}
		\norm{\left(I+A_p\right)^\alpha P_t}_{L^p\to L^p}\leq \frac{C_{\alpha,\nu,p}}{t^\alpha}, \quad t\in [0,T].
		\end{equation}
	\end{itemize} 
\end{proposition}

The action of $P_t$ on $f\in L^1$ is given by
\begin{equation*}
P_t f(x)=\int_D p_t(x,y)f(y)d y, \quad t>0,
\end{equation*}
where the \emph{heat kernel} $p_t(x,y)\geq 0,t>0,x,y\in \bar{D}$, is a smooth function, $p_t\in C^\infty(\bar D^2)$, $t>0$,
satisfying \eqref{eq:neumannproblem} with the
initial condition replaced by $u_t(x)\to \delta_y$ as $t\to 0$. 
The Neumann heat kernel can be controlled with the free heat kernel (on the whole plane) up to a multiplicative constant at the exponent, that is:
\begin{equation}\label{eq:kernelestimatePW}
        p_t(x,y)\leq \frac{C}{t\wedge |D|}e^{-|x-y|^2/(ct)},\quad x,y\in \bar D,
\end{equation}
with $C,c>0$ depending only on the domain $D$, \emph{cf.} \cite[Eq. (3.2)]{Wang2013}.
The following statement recalls the \emph{gradient estimates} on heat semigroup and kernel, we refer to \cite[Theorem1.2, Corollary 1.3, Lemma 3.1]{Wang2013} for a proof.

\begin{proposition}\label{prop:gradientestimates}
    There exist a constant $C>0$ depending only on $D$ such that,
    \begin{equation}\label{eq:gradientestimateBASE}
        \abs{\nabla P_t f(x)}\leq e^{Ct}P_t\pa{\abs{\nabla f}(x)},
        \quad t\geq 0,\,x\in D,\, f\in C^1_b(D).
    \end{equation}
    Moreover, there exist constants $C,c,c'>0$ depending only on $D$ such that
    \begin{equation}\label{eq:gradientestimatePW}
		\nabla p_t(x,y)\leq C (\one_{0<t\le 1} t^{-3/2}+\one_{t>1} e^{-c't})e^{-|x-y|^2/(ct)}.
	\end{equation}
	As a consequence, for all $t>0$ and $p\in (1,\infty)$,
	\begin{equation}\label{eq:gradientestimateOP}
	\norm{\nabla P_t}_{L^p\to C(\bar D)}\leq \frac{C_p}{(1\wedge t)^{1/2+1/p}}.
	\end{equation}
\end{proposition}
\noindent
As reported in \cite[Theorem 3.2]{Wang2013}, the latter implies the following:

\begin{corollary}\label{cor:riesztransform}
	For all $p> 1$, $\norm{\nabla (I+A_p)^{-1/2}}_{L^p\to L^p}\leq C_p$.
\end{corollary}
\noindent
We will also employ a gradient estimate for the heat semigroup acting on distribution spaces.

\begin{lemma}\label{lem:gradientdistributions}
    For all $t>0$, $\alpha>0$, $p\in (1,\infty)$,
    \begin{equation}
        \norm{\nabla P_t}_{\HH^{-\alpha,p}\to C(\bar D)}\leq \frac{C_{\alpha,\nu,p}}{t^{(\alpha+1)/2+1/p}}.
    \end{equation}
\end{lemma}

\begin{proof}
    Observe first that, by definition of $\HH^{-\alpha,p}$ and density of $\HH^{\alpha,p}$ in $L^p$, it holds for all $s>0$
    \begin{multline*}
        \norm{P_s}_{\HH^{-\alpha,p}\to L^p}
        =\sup_{f\in L^p} \frac{\norm{P_s (I+A_p)^{\alpha/2} f}_{L^p}}{\norm{f}_{L^p}}\\
        =\sup_{f\in \HH^{\alpha,p}} \frac{\norm{P_s (I+A_p)^{\alpha/2} f}_{L^p}}{\norm{f}_{L^p}}
        \leq \norm{(I+A_p)^{\alpha/2} P_s}_{L^p\to L^p} \leq \frac{C_{\alpha,\nu,p}}{s^{\alpha/2}},
    \end{multline*}
    the last step using \eqref{eq:contractivity}.
   For $0<s<t$ we can combine \eqref{eq:gradientestimateOP} and the estimate of above in order to obtain
   \begin{multline}
       \norm{\nabla P_t}_{\HH^{-\alpha,p}\to C(\bar D)}
       =\norm{\nabla P_{t-s} P_s  }_{\HH^{-\alpha,p}\to C(\bar D)}\\
       =\norm{\nabla P_{t-s}}_{L^p\to C(\bar D)} \norm{P_s}_{\HH^{-\alpha,p}\to L^p}
       \leq \frac{C_{\alpha,\nu,p}}{(t-s)^{1/2+1/p}s^{\alpha/2}},
    \end{multline}
    from which the thesis follows minimizing the right-hand side with respect to the parameter $s$.
\end{proof}

The following statement collects technical passages involving heat semigroups that will appear often in our computations, since they concern (duality) relations between operator norms of semigroups and uniform estimates of their kernels over the space domain $D$.

\begin{lemma}\label{lem:duality}
Let $\alpha\in \R$, $p\in (1,\infty)$, $t\ge0$, $\epsilon>0$.
For all $y\in \bar{D}$ and  $f\in\HH^{\alpha,p'}(D)\cap L^{p'}(D)$ it holds:
\begin{align}
\label{eq:heat_dual_representation}
&\int_D (I+A_p)^{\alpha/2}P_tp_\epsilon(\cdot,y)(x)f(x)dx \\ \nonumber
&\qquad= P_{t+\epsilon}(I+A_{p'})^{\alpha/2}f(y) =(I+A_{p'})^{\alpha/2}P_{t+\epsilon}f(y),\\
\label{eq:heat_dual_representation_grad}
&\int_D (I+A_p)^{\alpha/2}P_t\nabla_y p_\epsilon(\cdot,y)(x)f(x)dx \\
\nonumber
&\qquad= \nabla P_{t+\epsilon}(I+A_{p'})^{\alpha/2}f(y) =\nabla (I+A_{p'})^{\alpha/2}P_{t+\epsilon}f(y).
\end{align}
As a consequence
\begin{gather}
\sup_{y\in \bar{D}}\norm{P_tp_\epsilon(\cdot,y)}_{\HH^{\alpha,p}(D)} = \norm{P_{t+\epsilon}(I+A_{p'})^{\alpha/2}}_{L^{p'}(D)\to C(\bar{D})},\label{eq:heat_norm_sup}\\
\sup_{y\in \bar{D}}\norm{P_t\nabla_y p_\epsilon(\cdot,y)}_{\HH^{\alpha,p}(D)} = \norm{\nabla P_{t+\epsilon}(I+A_{p'})^{\alpha/2}}_{L^{p'}(D)\to C(\bar{D})}.\label{eq:heat_norm_grad_sup}
\end{gather}
\end{lemma}

\begin{proof}
    The first two claims follow from \eqref{eq:kernelestimatePW}. 
    For $y\in \bar{D}$ consider a sequence $y_n\in D$ such that $y_n\rightarrow y$;
    as for \eqref{eq:heat_dual_representation}, the following chain of equalities holds due to the regularity of $f$ and $p_{\epsilon}$:
    \begin{align*}
        \int_D (I+A_p)^{\alpha/2}P_tp_\epsilon(\cdot,y)(x)f(x)dx&=\int_D p_\epsilon(\cdot,y)(x)(I+A_{p'})^{\alpha/2}P_tf(x)dx\\ &= \int_D \lim_{n\rightarrow +\infty} p_\epsilon(\cdot,y_n)(x)(I+A_{p'})^{\alpha/2}P_tf(x)dx\\ & = \int_D \lim_{n\rightarrow +\infty} p_\epsilon(y_n,\cdot)(x)(I+A_{p'})^{\alpha/2}P_tf(x)dx\\ & =\lim_{n\rightarrow +\infty}\int_D  p_\epsilon(y_n,\cdot)(x)(I+A_{p'})^{\alpha/2}P_tf(x)dx\\ & = \lim_{n\rightarrow +\infty}P_{\epsilon}(I+A_{p'})^{\alpha/2}P_tf(y_n).
    \end{align*}
    This implies \eqref{eq:heat_dual_representation} thanks to the fact that $(I+A_{p'})^{\alpha/2}$ and the heat semigroup commute. The exchange between limit and integral is allowed thanks to \eqref{eq:kernelestimatePW}. 
    The proof of \eqref{eq:heat_dual_representation_grad} follows from a similar argument. Indeed,
    \begin{align*}
        &\int_D (I+A_p)^{\alpha/2}P_t\nabla_y p_\epsilon(\cdot,y)(x)f(x)dx \\
        &\qquad= \int_D \nabla_y p_\epsilon(\cdot,y)(x)(I+A_p)^{\alpha/2}P_tf(x)dx\\ 
        &\qquad =\nabla_y\int_D  p_\epsilon(\cdot,y)(x)(I+A_p)^{\alpha/2}P_tf(x)dx\\
        &\qquad =\nabla_y\int_D  \lim_{n\rightarrow +\infty}p_\epsilon(\cdot,y_n)(x)(I+A_p)^{\alpha/2}P_tf(x)dx\\ 
        &\qquad =\nabla_y\int_D  \lim_{n\rightarrow +\infty}p_\epsilon(y_n,\cdot)(x)(I+A_p)^{\alpha/2}P_tf(x)dx\\ 
        &\qquad = \nabla_{y}\lim_{n\rightarrow +\infty}P_{\epsilon}(I+A_{p'})^{\alpha/2}P_tf(y_n).
    \end{align*}
    \autoref{eq:heat_dual_representation_grad} follows from the regularity of $P_{\epsilon}(I+A_{p'})^{\alpha/2}P_t$ and the fact that $(I+A_{p'})^{\alpha/2}$ and the heat semigroup commute. After these preliminaries,  \eqref{eq:heat_norm_sup} and \eqref{eq:heat_norm_grad_sup} are easy to prove by duality. 
    Indeed, for each $y\in \bar D$,
    \begin{align*}
    \norm{P_tp_\epsilon(\cdot,y)}_{\HH^{\alpha,p}}& = \sup_{f\in \HH^{\alpha,p'}, \norm{f}_{L^{p'}}=1}\left\lvert\int_D (I+A_p)^{\alpha/2}P_t p_{\epsilon}(\cdot,y)(x)f(x) dx\right\rvert\\ & = \sup_{f\in \HH^{\alpha,p'}, \norm{f}_{L^{p'}}=1}\left\lvert (I+A_{p'})^{\alpha/2}P_{t+\epsilon}f(y)\right\rvert\\ & = \sup_{\norm{f}_{L^{p'}}=1}\left\lvert (I+A_{p'})^{\alpha/2}P_{t+\epsilon}f(y)\right\rvert.
    \end{align*}
    Considering the supremum of both sides for $y\in \bar{D}$ \eqref{eq:heat_norm_sup} follows immediately due to the fact that $(I+A_{p'})^{\alpha/2}P_{t+\epsilon}\in C(\Bar{D})$.
    \eqref{eq:heat_norm_grad_sup} is similar. Indeed
    \begin{align*}
        \norm{P_t\nabla_y p_\epsilon(\cdot,y)}_{\HH^{\alpha,p}}&=\sup_{f\in \HH^{\alpha,p'}, \norm{f}_{L^{p'}}=1}\left\lvert\int_D (I+A_p)^{\alpha/2}P_t \nabla_{y}p_{\epsilon}(\cdot,y)(x)f(x) dx\right\rvert\\ & = \sup_{f\in \HH^{\alpha,p'}, \norm{f}_{L^{p'}}=1}\left\lvert \nabla(I+A_{p'})^{\alpha/2}P_{t+\epsilon}f(y)\right\rvert\\ & = \sup_{\norm{f}_{L^{p'}}=1}\left\lvert \nabla(I+A_{p'})^{\alpha/2}P_{t+\epsilon}f(y)\right\rvert.
    \end{align*}
    \autoref{eq:heat_norm_grad_sup} then follows considering the supremum of both sides for $y\in \bar{D}.$
\end{proof}

We conclude the paragraph recalling
the regularizing property of Biot-Savart kernel,
following from the one of Green's function for Dirichlet boundaries,
for which we refer to \cite[Section 3]{grubb2016regularity}.

\begin{lemma}\label{lem:biotsavart}
    The linear operator $K[\omega]=-\nabla^\perp (-\Delta_{Dir})^{-1}\omega$,
    defined first for $\omega\in C^\infty(D)$, extends to continuous linear maps
    (still denoted by $K$)
    \begin{gather*}
        K: L^2(D) \to L^q(D;\R^2), \quad q\in [1,\infty),\\
        K: L^p(D) \to C(\bar D;\R^2), \quad p\in (2,\infty),\\
        K: H^{\alpha,p}(D)\to H^{\alpha+1,p}(D;\R^2), \quad p\in(1,\infty), \alpha\geq0.
    \end{gather*}
    Moreover, for all the above extensions $K[\omega]$ is divergence-less (in the sense of distributions),
    and its normal trace on $\de D$ (when defined) vanishes.
\end{lemma}

\subsection{\emph{Càdlàg} Functions and Aldous' Criterion}

For $T>0$, we denote by $\DD([0,T],S)$ the space of \emph{càdlàg} functions on $[0,T]$ taking values in a complete metric space $(S,d)$.
We will always endow $\DD([0,T],S)$ with the Skorohod metric:
we refer to \cite[Section 2.1]{Metivier1988} for a definition of the latter, 
which we will not be
using directly since we can rely on the following tightness criterion,
\cite[Theorem 3.2]{Metivier1988}.

\begin{proposition}[Aldous' Criterion]\label{prop:aldous}
    Consider a sequence of filtered probability spaces $(\Omega_n,\F_n,\PP_n)_{n\in\N}$ on each of which it is defined a 
    \emph{càdlàg} adapted process $(X^n_t)_{t\in [0,T]}$
    taking values in a complete separable metric space $(S,d)$.
    The laws of processes $(X^n_t)_{t\in [0,T]}$ 
    are tight on $\DD([0,T],S)$ if the following
    two conditions are satisfied:
    \begin{enumerate}
    \item for every $t$ in a dense subset of $[0,T]$ the laws of $X^n_t$ are tight on $S$;
    \item for all $\epsilon,\delta>0$ there exists $r_0>0$ and $n_0\in\N$ such that,
    for any sequence $(\tau_n)_{n\in\N}$ with $\tau_n$ a $\F_n$-stopping time, it holds
    \begin{equation}\label{eq:aldouscondition}
        \sup_{n\geq n_0}\sup_{r\in [0,r_0]} \PP_n\pa{d\pa{X^n_{(\tau_n+r)\wedge T},X^n_{\tau_n\wedge T}}>\delta}<\epsilon.
    \end{equation}
    \end{enumerate}
\end{proposition}

\subsection{Well-Posedness of Navier-Stokes Equations with Neumann Boundary}
In order to identify the limit dynamics of our interacting particle system with solutions of \eqref{eq:neuNS},
we will need a uniqueness result and \emph{a priori} estimates.
We set the discussion of the limit PDE in the space of \emph{càdlàg} functions since convergence of the particle system will take place in that space.

\begin{definition}\label{def:weaksol}
    Let $T>0$; $\omega\in \DD([0,T];L^2(D))$,
    is a weak solution of \eqref{eq:neuNS}
    if for all $\phi\in \HH^{2,2}(D)$, for $t\in [0,T]$,
    \begin{multline*}
        \int_D \phi \omega_t dx-\int_D \phi \omega_0 dx\\
        =\nu\int_0^t \int_D \Delta \phi\cdot \omega_s dx ds
        +\int_0^t \int_D \nabla \phi\cdot K[\omega_s]\omega_s dx ds
        +\nu\int_0^t \int_{\de D} \phi g_s d\sigma(x) ds.
    \end{multline*}
\end{definition}

\begin{remark}
    The latter is a good definition: thanks to \autoref{lem:biotsavart},
    integrability in space of the nonlinear term $\nabla \phi\cdot K[\omega]\omega$ 
    follows from H\"older's inequality, since both $\nabla\phi$ and $K[\omega]$ are $L^q$ for all $q>2$.
\end{remark}

\begin{proposition}\label{prop:H-1method}
    Given $\omega_0\in L^2$ and $g\in L^2([0,T],L^2)$,
    there exists a unique weak solution of \eqref{eq:neuNS} in the sense of \autoref{def:weaksol} with $\omega(0)=\omega_0$.
    Moreover, the unique solution belongs to $C([0,T];L^2(D))\cap L^2([0,T];\HH^{1,2}(D))$.
\end{proposition}

The same statement holds for the cutoff PDE \eqref{eq:cneuNS}
for which in fact, thanks to the bounded nonlinearity, much stronger
well-posendenss results can be obtained.

\begin{proof}
    Consider two weak solutions $\omega, \tilde \omega\in \DD([0,T];L^2(D))$, let $w_t=\omega_t-\tilde \omega_t$. 
    With standard passages (we refer for instance to \cite[Section 3]{flandoli2022heat}) one can extend the weak formulation of the PDE to time-dependent tests 
    \begin{equation*}
        \phi\in C^1([0,T]; L^{2})\cap C([0,T];\HH^{2,2})
    \end{equation*}
    (this makes use of the right continuity hypothesis on $\omega$),
    then, taking $\phi_t(s,x)=P_{t-s}\psi(x),\ \psi\in \HH^{2,2}$ transforms the PDE into the variation-of-constants form
    \begin{equation*}
    \omega_t = P_t\omega_0 -\int_0^t P_{t-s}(\div(K[\omega_s]\omega_s))ds
    +\int_0^t (I-\nu\Delta)P_{t-s} N[g_s]ds.
    \end{equation*}
    A similar formula holds for $\tilde\omega_t$. Therefore $w_t$ satisfies
    \begin{equation*}
        w_t=-\int_0^t P_{t-s}(\div(K[w_s]\tilde\omega_s))ds-\int_0^t P_{t-s}(\div(K[\omega_s]w_s))ds.
    \end{equation*}
    By Sobolev embedding, for all $q>2$,
    \begin{align*}
        K[w_s],\ K[\omega_s]\in \DD([0,T];L^q(D;\R^2)),
    \end{align*}
    hence, by Holder inequality, again for all $q>2$,
    \begin{align*}
        K[w_s]\tilde\omega_s,\ K[\omega_s]w_s\in \DD\left([0,T];L^{\frac{q+2}{2q}}(D;\R^2)\right).
    \end{align*}
    We can thus control nonlinear terms by means of ultracontractivity \eqref{eq:ultracontractivity},
    H\"older inequality and \autoref{lem:biotsavart}:
    \begin{align*}
       \norm{w_t}_{L^2}&
        \lesssim_{\nu,T,q} \int_0^t \frac{1}{(t-s)^{\frac{q+2}{2q}-\frac{1}{2}}}\norm{P_{t-s}(\div(K[w_s]\tilde\omega_s))}_{L^{\frac{2q}{q+2}}}ds\\ 
        &\quad +\int_0^t \frac{1}{(t-s)^{\frac{q+2}{2q}-\frac{1}{2}}}\norm{P_{t-s}(\div(K[\omega_s]w_s))}_{L^{\frac{2q}{q+2}}}ds\\ 
        & \lesssim_q \int_0^t \frac{1}{(t-s)^{\frac{q+2}{2q}}}\left(\norm{K[w_s]\tilde\omega_s}_{L^{\frac{2q}{q+2}}}+\norm{K[\omega_s]w_s}_{L^{\frac{2q}{q+2}}}\right)ds\\ 
        & \lesssim_q \int_0^t \frac{\norm{\omega_s}_{L^2}+\norm{\tilde \omega_s}_{L^2}}{(t-s)^{\frac{q+2}{2q}}} \norm{w_s}_{L^2} ds.
    \end{align*}
The uniqueness statement now follows from Gr\"onwall's Lemma.

The last statement of the Proposition follows from existence of 
   a weak solution belonging to $C([0,T];L^2(D))\cap L^2([0,T];H^{1,2}(D))$.
   This can be proved with a standard Galerkin approximation together with
   the usual energy estimate (see the beginning of the proof of \autoref{prop:pdeapriori} below),
   we refer to \cite[Theorem 2.6]{Jourdain2004} for details.
\end{proof}

\begin{proposition}\label{prop:pdeapriori}
    Any weak solution of \eqref{eq:neuNS} in the sense of \autoref{def:weaksol} in $\mathbb{D}([0,T];L^p)$
    with $\omega_0\in L^p$, $p>2$, satisfies
    \begin{equation}\label{eq:aprioriestimate}
        \norm{K[\omega]}_{L^\infty([0,T]; L^\infty)}
        \lesssim_{T,p,\nu}\norm{\omega_0}_{L^p}+\norm{\omega_0}^2_{L^2}
        +\norm{g}_{L^2([0,T]; L^p)}+\norm{g}^2_{L^2([0,T]; L^2)}.
    \end{equation}

    As a consequence, there exists $M_0=M_0(\omega_0,g)$ such that for all
    $M\geq M_0$ the unique weak solution of \eqref{eq:neuNS} coincides with
    the one of the cutoff PDE \eqref{eq:cneuNS} with the same initial and boundary
    data.
\end{proposition}

In order to establish this \emph{a priori} estimate we first recall 
a convenient representation for boundary terms 
appearing in the weak formulation of \eqref{eq:neuNS}. 

\begin{lemma}\label{lem:neumannop}
    Let $g\in L^p(\de D)$, $p\geq 2$ and $\epsilon>0$. There exists a unique weak solution $u\in H^{1+\frac{1}{p}-\epsilon,p}(D)$ of 
    \begin{equation*}
        \begin{cases}
            (I-\nu\Delta)u(x)=0, &x\in D,\\
            \hat n\cdot \nabla u(x)=g(x), &x\in \de D,
        \end{cases}
    \end{equation*}
    that is, for all $\phi\in H^{1,p'}(D)$,
    \begin{equation*}
        \int_D (\phi u +\nu\nabla \phi \cdot \nabla u)dx=\nu\int_{\de D} \phi g d\sigma(x).
    \end{equation*}
    Moreover, the solution map $N[g]=u$ defines a continuous linear operator $N:L^p(\de D)\to  \HH^{1+1/p-\delta,p}$ for any $\delta>0$.
\end{lemma}

We refer to \cite[Chapter 4]{trie1983fun} for the (standard) proof.
It follows that
\begin{equation*}
    \int_ {\de D} \phi g d\sigma(x)=\int_D (I+A_{p'})\phi N[g] dx,
    \quad \phi\in H^{2,p'}_{Neu}, \, g\in L^p(\de D).
\end{equation*}

\begin{proof}[Proof of \autoref{prop:pdeapriori}]
    For the sake of completeness we sketch the classical argument 
    for the \emph{a priori} estimate in the $L^2$ setting.
    By \autoref{prop:H-1method} we know that the weak solution
    actually belongs to $C([0,T];L^2(D))\cap L^2([0,T];H^{1,2}(D))$,
    so we can evaluate
    \begin{equation*}
    \frac12\frac{d}{dt}\norm{\omega_t}_{L^2}^2
    =-\nu\norm{\nabla \omega_t}_{L^2}^2+\int_D \nabla\omega_t\cdot K[\omega_t]\omega_t dx+\nu\int_{\partial D} \omega_t g_t d\sigma(x).
    \end{equation*}
    The second term on the right-hand side vanishes, and the third one can be controlled thanks to the
    properties of the trace of functions in $H^1(D)$:
    \begin{align*}
    \int_{\partial D} \omega_t g_t d\sigma(x)&\leq C\pa{\norm{\omega_t}_{L^2}+\norm{\nabla\omega_t}_{L^2}}\norm{g_t}_{L^2(\partial D)}\\ 
    &\leq  \frac{1}{2}\norm{\omega_t}_{L^2}^2+\frac{1}{2}\norm{\nabla\omega_t}_{L^2}^2+C\norm{g_t}_{L^2(\partial D)}^2.    
    \end{align*}
    Therefore, we obtain 
    \begin{equation*}
    \frac12\frac{d}{dt}\norm{\omega_t}_{L^2}^2+\frac\nu2\norm{\nabla \omega_t}_{L^2}^2\leq \frac{\nu}{2}\norm{\omega_t}_{L^2}^2+C\nu\int_{\partial D} \omega_t g_t d\sigma(x),
    \end{equation*}
    from which by Gr\"onwall's lemma we conclude that
    \begin{equation}\label{eq:aprioriL2}
        \norm{\omega}^2_{L^\infty([0,T],L^2)}+\norm{\nabla\omega}^2_{L^2([0,T],L^2)}
        \lesssim_{\nu,T} \norm{\omega_0}^2_{L^2}+\norm{g}^2_{L^2([0,T],L^2(\de D))}.
    \end{equation}

    Arguing as in the proof of \autoref{prop:H-1method}, extending the weak formulation of the PDE to time-dependent tests $$\phi\in C^1([0,T]; L^{p'})\cap C([0,T];\HH^{2,p'})$$
    and taking $\phi_t(s,x)=P_{t-s}\psi(x),\ \psi\in \HH^{2,p'}$, standard passages
    allow to rewrite the PDE into the variation-of-constants form
    \begin{equation}\label{eq:vNSduhamel}
    \omega_t = P_t\omega_0 -\int_0^t P_{t-s}(K[\omega_s]\cdot \nabla\omega_s)ds
    +\int_0^t (I-\nu\Delta)P_{t-s} N[g_s]ds,
    \end{equation}
    from which
    \begin{equation*}
        \norm{\omega_t}_{L^p}
        \leq \norm{\omega_0}_{L^p}+\int_0^t \norm{P_{t-s}(K[\omega_s]\cdot \nabla\omega_s)}_{L^p}ds
        +\int_0^t \norm{(I-\nu\Delta)P_{t-s} N[g_s]}_{L^p}ds.
    \end{equation*}
    We control the nonlinear term by means of ultracontractivity \eqref{eq:ultracontractivity},
    H\"older inequality and \autoref{lem:biotsavart}: for $\frac{2p}{p+2}<q<2<p$ it holds
    \begin{align*}
        \int_0^t \norm{P_{t-s}(K[\omega_s]\cdot \nabla\omega_s)}_{L^p}ds
        &\leq \int_0^t (t-s)^{1/p-1/q} \norm{(K[\omega_s]\cdot \nabla\omega_s)}_{L^q}ds\\
        &\leq \int_0^t (t-s)^{1/p-1/q} \norm{K[\omega_s]}_{L^{2q/(2-q)}} \norm{\nabla\omega_s}_{L^2}ds\\
        &\lesssim_q \norm{\omega}_{L^\infty([0,T],L^2)} \int_0^t (t-s)^{1/p-1/q} \norm{\nabla\omega_s}_{L^2} ds\\
        &\lesssim_q \norm{\omega}_{L^\infty([0,T],L^2)} \norm{\omega}_{L^2([0,T],H^{1,2}(D))},
    \end{align*}
    the last step following from H\"older inequality and the fact that
    $\int_0^t (t-s)^{2/q-2/p}ds<\infty$.
    As for the boundary term, by \autoref{prop:fractionalpowers}, \autoref{lem:neumannop} and \eqref{eq:contractivity}, for $\delta>0$,
    \begin{align*}
        &\int_0^t \norm{(I-\nu\Delta)P_{t-s} N[g_s]}_{L^p}ds\\
        &\quad = \int_0^t \lVert(I+A_p)^{\frac{1}{2}+\delta-\frac{1}{2p}}P_{t-s}(I+A_p)^{\frac{1}{2}-\delta+\frac{1}{2p}}N\left[g_s\right] \rVert_{L^p} ds \\
        &\quad \lesssim_{p,\delta} \int_0^t (t-s)^{-1/2-\delta+1/(2p)}
         \lVert(I+A_p)^{\frac{1}{2}-\delta+\frac{1}{2p}}N\left[g_s\right] \rVert_{L^p} ds \\
        &\quad \leq \left(\int_0^t (t-s)^{-1-2\delta+1/p}ds\right)^{1/2}\lVert g\rVert_{L^2(0,T;L^p)},
    \end{align*}
    where the integral in parentheses is finite if we choose $\frac{1}{2p}>\delta$.
    
    By \autoref{lem:biotsavart}, and combining the estimates above, we deduce that
    \begin{multline*}
        \norm{K[\omega]}_{L^\infty([0,T], L^\infty)}\lesssim_p 
        \sup_{t\in [0,T]}\norm{\omega_t}_{L^p}\\
        \lesssim_{T,p}\norm{\omega_0}_{L^p}+\norm{\omega_0}^2_{L^2}
        +\norm{g}_{L^2([0,T]; L^p)}+\norm{g}^2_{L^2([0,T]; L^2)},
    \end{multline*}
    from which the first statement of the Proposition follows.

    As for the second statement, it suffices to take
    \begin{equation*}
        M\geq C_{T,p,\nu}\left(\norm{\omega_0}_{L^p}+\norm{\omega_0}^2_{L^2}
        +\norm{g}_{L^2([0,T]; L^p)}+\norm{g}^2_{L^2([0,T]; L^2)}+1\right),
    \end{equation*}
    where $C_{T,p,\nu}$ is the constant implied in the previous inequality, 
    making it so that the nonlinear term of solutions of \eqref{eq:neuNS}
    and \eqref{eq:cneuNS} coincide by definition of the cutoff $F$. 
\end{proof}

\subsection{Stochastic Integrals and Burkholder-Davis-Gundy Inequalities}
Let $V$ be a separable Hilbert space, and consider a cylindrical Wiener process $W$
on $V$ with identity covariance operator; for our purposes $V$ will in fact
always be finite-dimensional.
Let $(\Omega,\F,\PP)$ be the standard filtered probability space on which $W$ is defined (and adapted), and consider a progressively measurable stochastic process
$(\psi_t)_{t\in [0,T]}$ taking values into $HS(V,L^2)$, the space of
Hilbert-Schmidt operators between $V$ and $L^2=L^2(D)$.
Let us also assume that for all $t\in [0,T]$, $\PP$-almost surely,
\begin{equation*}
    \int_0^t \norm{\psi_s}_{HS(V,L^2)}^2ds<\infty.
\end{equation*}
We refer to the classical monograph \cite{DaPrato2014} for the definition of
the stochastic integral $\int_0^t\psi_s dW_s$ (a continuous $L^2$-valued local martingale) and its basic properties.

The following collects (two versions of) the Burkholder-Davis-Gundy inequality,
on which we will rely to bound uniformly
in time terms due to the Brownian noise in the particle dynamics.

\begin{proposition}\label{prop:bdg}
    Under the assumptions of above, for any $p>0$ it holds
    \begin{equation}\label{eq:bdg}
        \expt{\sup_{t\in[0,T]}\norm{\int_0^t\psi_s dW_s}_{L^2}^p}
        \lesssim_p \expt{\pa{\int_0^T \norm{\psi_s}_{HS(V,L^2)}^2ds}^{p/2}}.
    \end{equation}
    Moreover if $(S_t)_{t\geq 0}$ is an analytic semigroup on $L^2$, for any $p>2$ it holds
    \begin{equation}\label{eq:bdgconv}
        \expt{\sup_{t\in[0,T]}\norm{\int_0^t S_{t-s}\psi_s dW_s}_{L^2}^p}
        \lesssim_p \expt{\pa{\int_0^T \norm{\psi_s}_{HS(V,L^2)}^2ds}^{p/2}}.
    \end{equation}
\end{proposition}

The proof can be found in \cite{Seidler2010},
to which we refer for a detailed and more general discussion
on this kind of inequalities.

\section{Definition of the Model and Main Results}\label{sec:statements}

Since we are considering a finite, fixed time horizon $T>0$, for the sake of lightening notation we assume from now, without loss of generality, $T=1$.

\subsection{Generation of Particles}\label{ssec:grid}
Let us introduce a grid of mesh $1/n$ spanning $(\set{0}\times D)\cup ((0,1]\times \de D)$.
Let $\gamma:S^1\simeq (0,1]\to \de D$ be a diffeomorphism parametrizing the boundary of $D$;
for $n\geq 1$ we set
\begin{gather*}
    L^n= \set{0}\times L^n_{in} \cup \set{\frac1{2n},\frac3{2n}\dots \frac{2n-1}{2n}}\times L^n_{bd},\\
    L^n_{in}= \pa{\Z/n}^2\cap D, \quad L^n_{bd}= \set{\gamma\pa{h/n}\mid h=1,\ldots n}.
\end{gather*}
We will write 
\begin{equation}
    i\mapsto(t_i^n,\zeta_i^n)\in L^n, \quad i=1,\dots,N(n),
\end{equation}
to enumerate the points of the grid $L^n$, $N(n)$ being the cardinality of $L^n$.
From now on we simply write $N=N(n)$ implying the dependence on $n$. We now introduce a partition\footnote{To be precise, a measurable partition, since boundaries of the $Q^n_i$'s overlap.} of $(\set{0}\times D)\cup ((0,1]\times \de D)$ whose elements are centered at points $\zeta_i$:
\begin{equation}
    Q_i^n=
    \begin{cases}
    \pa{\zeta_i^n+[-\frac1{2n},\frac1{2n}]^2}\cap D, &t_i^n=0,\\
    [t_i^n-\frac1{2n},t_i^n+\frac1{2n}]\times \gamma\pa{\gamma^{-1}(\zeta_i)+[-\frac1{2n},\frac1{2n}]}, &t_i^n>0.
    \end{cases}
\end{equation}
Notice that $N(n)\simeq n^2$ and the area of $Q_i^n$ (both if $t_i^n=0$ or not) is of order $n^{-2}$.
We denote by
\begin{equation}
    A^n(t)=\set{i:0\leq t_i^n\leq t}, \quad A_0^n(t)=\set{i:0< t_i^n\leq t},
\end{equation}
the set of indices relative to particles created before time $t$.
Given $\omega_0\in \B_b(D) $, $g\in \B_b([0,1]\times \de D)$
with $g,\omega_0\geq 0$, we set
\begin{equation}
    \omega_i^n=
    \begin{cases}
    \int_{Q_i^n}\omega_0(x)dx &t_i^n=0,\\
    \int_{Q_i^n} g_s(x) dsd\sigma(x) &t_i^n>0.
    \end{cases}
\end{equation}
From our assumptions on the mesh of the grid, $\omega_0$ and $g$ it follows that 
\begin{equation*}
    \omega_i^n\lesssim \frac{1}{n^2}\quad \forall n\in\N, \ i=1,\dots N.
\end{equation*}

\subsection{Diffusion Processes and Reflecting Boundaries}\label{ssec:BMreflecting}

Let $(\Omega,\F,\PP)$ a complete, filtered probability space satisfying the standard hypothesis, on which it is defined a sequence $(B^i)_{i\in \N}$ of independent $\F$-Brownian motions.
We have the following (probabilistically strong) well-posedness result:

\begin{proposition}\label{prop:wellposSDE}
    For all $n\geq 1$, on $(\Omega,\F,\PP)$, there exists a unique continuous $\bar D^N$-valued adapted process $x^n(t)=(x_1^n(t),\dots, x_N^n(t))_{t\in [0,T]}$
    and a continuous $\R^{2\times N}$-valued adapted process $k^n(t)=(k_1^n(t),\dots, k_N^n(t))_{t\in [0,T]}$ with bounded variation trajectories
    such that: for $i=1,\dots,N$, $t\in [t_i^n,1]$,
    \begin{gather*}
        x_i^n(t)-\zeta_i^n=\int_{t_i^n}^t F\pa{\sum_{j\in A^n(s)} \omega_j^n K_n(x_i^n(s),x_j^n(s))}ds + \sqrt{2\nu}\int_{t_i^n}^t dB^i_s- k_i^n(t),\\
        k_i^n(t)=\int_{t_i^n}^t\hat n(x_i^n(s))d|k_i^n|(s), 
        \quad |k_i^n|(t)=\int_{t_i^n}^t\one_{x_i^n(s)\in \de D}d|k_i^n|(s),
    \end{gather*}
    while for $t<t_i^n$ we impose $k_i^n(t)=0$ and $x_i^n(t)=\zeta_i^n$.
\end{proposition}

The proof can be straightforwardly adapted from the one given in \cite{Lions1984,Sznitman1984} for systems of SDEs with regular coefficients and reflecting boundaries: the only difference is generation of new particles at the boundary, which is taken care of applying the well-posedness result on each
time interval of length $1/n$ during which particles are not generated.
We also refer to \cite[Section 4]{Jourdain2004}
for details on a SDE system (closely related to ours) 
including boundary generation at random times.
It\^o's formula for the process $x^n(t)$ takes the following form (for which we refer again to \cite{Lions1984}):

\begin{corollary}\label{cor:itoSDE}
    If $\phi\in C^2(\bar D)$ with $\nabla\phi\cdot \hat n=0$ on $\de D$, for $i=1,\dots,N$, $t\in [t_i^n,T]$,
    \begin{multline*}
        d\phi(x_i^n(t))=F\pa{ \sum_{j\in A^n(t)}\omega_j^n K_n(x_i^n(t),x_j^n(t))}\cdot \nabla\phi(x_i^n(t))dt\\
        +\nu \Delta \phi(x_i^n(t))dt+\sqrt{2\nu}\nabla\phi(x_i^n(t))dB^i_t.
    \end{multline*}
\end{corollary}

Notice that the hypothesis $\nabla\phi\cdot \hat n=0$ on $\de D$ makes it so that reflection terms $-k^n_i$
do not appear in the It\^o formula. In our applications, this assumption will be verified thanks to our
choice of regularizing the empirical measure with the heat kernel under Neumann bouondary conditions.

\begin{remark}\label{rmk:pointvortex}
    The classical vortex dynamics in a bounded domain
    is a system of singular ODEs of the form
    \begin{gather*}
        \dot x_i=\sum_{j\neq i} \xi_j K(x_i,x_j) + \xi_i \nabla^\perp \gamma(x_i,x_i),
        \quad x_i\in D,\, \xi_i\in \R^\ast,
        \quad i,j=1,\dots N,\\
        G=(-\Delta_{Dir})^{-1}, \quad \gamma(x,y)=G(x,y)+\frac1{2\pi}\log\abs{x-y},
    \end{gather*}
    including self-interaction terms induced from the boundary effects;
    this is necessary for the system to satisfy in a weak form
    (as in \cite{Schochet1996}) the 2-dimensional Euler equations.
    We refer to \cite[Chapter 4]{Marchioro1994} for a
    general introduction to the topic, to \cite{Durr1982,Grotto2020a,Grotto2022a,Grotto2022b}
    on the issue of well-posedness of the dynamics and vortex collisions,
    and to \cite{Lions1998,Grotto2020b,Grotto2020c} on the statistical mechanics
    point of view.
    
    Self-interactions diverge logarithmically at $\de D$, 
    and this prevents us to include them
    in our model because the Brownian part of the dynamics 
    (which we must include to model viscosity) might
    drive particles onto the boundary causing blow-up of the dynamics at finite time.
    Nevertheless, under the Mean-Field scaling of particle intensities we are considering, self-interactions should be negligible in macroscopic limit.
    Indeed (heuristically) a self-interaction term $\omega_i \nabla^\perp \gamma(x_i,x_i)$ in our model would scale as $n^{-2}$, while the nonlinear and noise terms are of order 1 as $n\to\infty$.
    Hence, self-interactions due to the boundary appear to be irrelevant for the purpose of our discussion.
\end{remark}

\subsection{Convergence to Navier-Stokes Equations}\label{ssec:mainresult}
The following is the main result of the paper.

\begin{theorem}\label{thm:main}
    Let $p>2$, $\frac{2}{p}<\alpha<1$, and $\epsilon=\epsilon(n)\gtrsim n^{-1/2}$.
    Assume that $\omega_0\in \lip(\bar D) $ and $g\in \B_b([0,1]\times \de D)$, $\omega_0,g\geq 0$, and let the related notation introduced above prevail.
    There exists $M>0$ (only depending on $\omega_0,g$) such that
    for every $\eta\in \pa{\frac{2}{p},\alpha}$,
    as $n\to \infty$,
    the kernel-smoothed empirical measure 
    $\omega^n_t=\sum_{i\in A^n(t)}\omega_i^n p_\epsilon(\cdot,x_i^n(t))$ 
    converges in probability on $\DD([0,1],\HH^{\eta,p})$ 
    to the unique weak solution (given by \autoref{prop:H-1method})
    of \eqref{eq:neuNS} with initial datum $\omega_0$ and boundary source $g$.
\end{theorem}

By Morrey's inequality, the stated convergence implies that on 
$\DD([0,1],C(\bar D))$.

\section{Uniform Bounds}\label{sec:uniformbounds}

The proof of \autoref{thm:main} essentially relies on uniformly bounding (in $n$)
the approximating process $\omega^n(t,x)$ in terms of strong norms,
allowing the application of Aldous' tightness criterion, \autoref{prop:aldous},
and to pass to the limit the dynamics obtaining \eqref{eq:neuNS}.
The way we exploit the regularizing effect of the Brownian noise driving particles,
that is the parabolic nature of the corresponding PDE dynamics, consists in formulating
the evolution problem in Duhamel form (that is the variation-of-constants form, or \emph{mild} formulation), in complete analogy with the \emph{a priori} estimates for the limiting PDE in the previous section.

\subsection{Duhamel Formulation of Empirical Measure Dynamics}
We adopt, for the remainder of the Section, the notation of \autoref{ssec:BMreflecting}.
Let thus $(x_1^n,\dots,x_N^n)$ be the $\bar D^N$-valued stochastic process on $[0,1]$ defined by \autoref{prop:wellposSDE}
as the unique strong solution of the approximating particle system, and 
\begin{equation*}
    S^n_t= \sum_{i\in A^n(t)}\omega_i^n \delta_{x_i^n(t)}.
\end{equation*}
A direct application of the It\^o formula in \autoref{cor:itoSDE} shows that,
if $\phi\in C^2(\bar D)$ with $\nabla\phi\cdot \hat n=0$ on $\de D$,
for all $t\in [0,1]$ it holds
    \begin{align}\label{eq:itoS}
	\int_D\phi dS^n_t-\int_D\phi dS^n_0
	&=\sum_{i\in A^n_0(t)}\omega_i \phi(\zeta_i) 
    +\nu \int_0^t \int_D \Delta\phi dS^n_s ds\\ \nonumber
    &\quad +\int_0^t\int_D F\pa{\int_D K_n(x,y)dS^n_s(y)}\nabla\phi(x)dS^n_s(x)ds\\ \nonumber
	&\quad +\sqrt{2\nu} \int_{0}^t \sum_{i\in A^n(s)}\omega_i \nabla\phi(x_i^n(s))dB^i_s.
    \end{align}
Let us now considered the kernel-smoothed empirical measure
\begin{equation*}
    \omega^n_t(x)=P_\epsilon S^n_t= 
    \sum_{i\in A^n(t)}\omega_i^n p_\epsilon(x,x_i^n(t)),\quad x\in \bar D,
\end{equation*}
which, in sight of \eqref{eq:itoS}, must satisfy for all $t\in [0,1]$ and $x\in \bar D$
\begin{align} \label{eq:itoOmega}
        \omega^n_t(x)-\omega^n_0(x) 
            &= \sum_{i\in A^n_0(t)}\omega_i^n p_\epsilon(x,\zeta_i^n)
            +\nu \int_0^t \Delta \omega^n_s(x)ds\\ \nonumber
			&\quad +\int_0^t \int_D F\pa{\int_D K_n(y,z)dS^n_s(z)}\nabla_y p_\epsilon(x,y)dS^n_s(y)ds\\ \nonumber
			&\quad  +\sqrt{2\nu} \int_{0}^t \sum_{i\in A^n(s)}\omega_i^n \nabla_y p_\epsilon(x,x_i^n(s))dB^i_s.
\end{align}
As in the proof of \autoref{prop:pdeapriori}, we can derive from the latter formulation of the dynamics the variation-of-constants form with standard passages.
For all $t\in [0,1]$ and $x\in \bar D$ it holds:
\begin{align} \label{eq:duhamelformula}
        \omega^n_t(x)
            &=P_t\omega^n_0(x) +\sum_{i\in A^n_0(t)}\omega_i^n P_{t-t_i^n} p_\epsilon(x,\zeta_i^n)\\ \nonumber
			&\quad +\int_0^t P_{t-s} \int_D F\pa{\int_D K_n(y,z)dS^n_s(z)}\nabla_y p_\epsilon(x,y)dS^n_s(y)ds\\ \nonumber
			&\quad +\sqrt{2\nu} \int_{0}^t \sum_{i\in A^n(s)}\omega_i^n P_{t-s} \nabla_y p_\epsilon(x,x_i^n(s))dB^i_s.   
\end{align}

\subsection{Preliminary Estimates on Particle Creation and Diffusion}

We begin by estimating separately the terms due to particle generation at $t=0$ and at the boundary at later $t>0$.

\begin{lemma}\label{lem:initialdatum}
    Let $\epsilon=\epsilon(n)\gtrsim n^{-2/(3+\alpha)}\gtrsim n^{-1/2}$;
    for all $p>2$, $\frac{2}{p}<\alpha<1$ it holds
    \begin{equation*}
        \norm{\sum_{i: t_i^n=0}\omega^i_n p_{\epsilon}(\cdot,\zeta^n_i)}_{\HH^{\alpha,p}}\leq C_{\alpha,p}.
    \end{equation*}
\end{lemma}

The idea of the proof is that, for $n$ large and $\epsilon$ small,
\begin{align*}
\sum_{i: t_i^n=0}\omega_i^n p_\epsilon(x,\zeta_i^n) \approx \int_{D} p_\epsilon(x,y) \omega_0(y) dy,
\end{align*}
and the right-hand side is bounded if, say, $\omega_0$ is Lipschitz continuous.

\begin{proof}
We have
\begin{multline*}
\norm{\sum_{i: t_i^n=0}\omega^i_n p_{\epsilon}(\cdot,\zeta^n_i)}_{\HH^{\alpha,p}}
\le \norm{\int_{D} p_\epsilon(\cdot,y) \omega_0(y) dy}_{\HH^{\alpha,p}}\\
\quad +\norm{\int_{D} p_\epsilon(\cdot,y) \omega_0(y) dy -\sum_{i: t_i^n=0}\omega^i_n (I+A_p)^{\alpha/2}p_{\epsilon}(\cdot,\zeta^n_i)}_{\HH^{\alpha,p}}
\quad =:I_{\omega_0} +R.
\end{multline*}
We bound the term $I_{\omega_0}$:
\begin{equation}\label{eq:bd_I_omega_0}
I_{\omega_0}= \norm{(I+A_p)^{\alpha/2}P_\epsilon\omega_0}_{L^p}
= \norm{P_\epsilon(I+A_p)^{\alpha/2}\omega_0}_{L^p}
\le \|\omega_0\|_{\HH^{\alpha,p}}.
\end{equation}
Concerning $R$, recall that, for $i$ such that $t_i^n=0$, $\omega_i^n= \int_{Q^n_i} \omega_0(y)dy$. Hence,
\begin{align*}
&\int_{D} p_\epsilon(\cdot,y)(x) \omega_0(y) dy -\sum_{i: t_i^n=0}\omega^i_n p_{\epsilon}(\cdot,\zeta^n_i)(x)\\
&= \sum_{i: t_i^n=0} \int_{Q^n_i} p_\epsilon(\cdot,y)(x)-p_\epsilon(\cdot,\zeta_i^n)(x) \omega_0(y) dy\\
&= \sum_{i: t_i^n=0} \int_{Q^n_i} (y-\zeta_i^n)\cdot \int_0^1 \nabla_y p_\epsilon(\cdot,\zeta_i^n+\xi(y-\zeta_i^n))(x) d\xi\, \omega_0(y)dy
\end{align*}
(note that $\zeta_i^n+\xi(y-\zeta_i^n)$ belongs to $\bar{D}$ since $\bar{D}$ is convex). Recall that, for every $i$, $|y-\xi^n_i|\lesssim 1/n$ on $Q_i^n$. We then have, by \eqref{eq:heat_norm_grad_sup},
\begin{align*}
R&\le \sum_{i: t_i^n=0} \int_{Q^n_i} (y-\zeta_i^n)\cdot \int_0^1 \norm{\nabla_yp_\epsilon(\cdot,\zeta_i^n+\xi(y-\zeta_i^n))}_{\HH^{\alpha,p}} d\xi\, \omega_0(y)dy\\
&\lesssim \frac{1}{n} \int_{D} \omega_0(y)dy \, \sup_{y\in \bar{D}} \norm{\nabla_y p_\epsilon(\cdot,y)}_{\HH^{\alpha,p}}\\
&= \frac{1}{n} \int_{D} \omega_0(y)dy \, \norm{\nabla P_\epsilon (I+A_{p'})^{\alpha/2}}_{L^{p'}\to C(\bar{D})}.
\end{align*}
Thanks to the gradient bound Lemma \ref{lem:gradientdistributions}, we get
\begin{equation}\label{eq:bd_R_omega_0}
R\lesssim_{\nu,p} \frac{1}{n} \epsilon^{-1/2-\alpha/2-1/p'} \|\omega_0\|_{L^1}
\le \frac{1}{n} \epsilon^{-3/2-\alpha/2} \|\omega_0\|_{L^1}.
\end{equation}
Putting together the bounds \eqref{eq:bd_I_omega_0} and \eqref{eq:bd_R_omega_0}, we conclude that
\begin{align*}
\norm{\sum_{i: t_i^n=0}\omega^i_n p_{\epsilon}(\cdot,\zeta^n_i)}_{\HH^{\alpha,p}}\lesssim_{\nu,p} \|\omega_0\|_{\HH^{\alpha,p}} +\frac1n \epsilon^{-3/2-\alpha/2} \|\omega_0\|_{L^1}.
\end{align*}
By the assumption $\epsilon=\epsilon(n)\gtrsim n^{-2/(3+\alpha)}$, we get the desired bound.
\end{proof}

\begin{lemma}\label{lem:boundarygeneration}
    Let $\epsilon=\epsilon(n)\gtrsim n^{-2/(2+\alpha)}\gtrsim n^{-2/3}$;
    for all $p>2$, $\frac{2}{p}<\alpha<1$ it holds
    \begin{equation*}
        \sup_{t\in [0,1]}\norm{\sum_{i\in A^n_0(t)}\omega_i^n P_{t-t_i^n} p_\epsilon(\cdot,\zeta_i^n)}_{\HH^{\alpha,p}}\leq C_{\alpha,p}.
    \end{equation*}
\end{lemma}
\noindent
As for the previous lemma, the idea of the proof is that, for $n$ large and $\epsilon$ small,
\begin{align*}
\sum_{i\in A^n_0(t)}\omega_i^n P_{t-t_i^n} p_\epsilon(\cdot,\zeta_i^n)(x) \approx \int_0^t \int_{\partial D} P_{t-s} p_\epsilon(x,y) g_s(y) dyds,
\end{align*}
and the right-hand side is bounded if, say, $g$ is bounded.

\begin{proof}
We have
\begin{multline*}
\norm{\sum_{i\in A^n_0(t)}\omega_i^n P_{t-t_i^n} p_\epsilon(\cdot,\zeta_i^n)}_{\HH^{\alpha,p}}
\le \norm{\int_0^t\int_{\partial D}  P_{t-s} p_\epsilon(\cdot,y) g_s(y) dsd\sigma(y)}_{\HH^{\alpha,p}}\\
+\norm{\int_0^t\int_{\partial D}  P_{t-s} p_\epsilon(\cdot,y) g_s(y) dsd\sigma(y) -\sum_{i\in A^n_0(t)}\omega_i^n P_{t-t_i^n} p_\epsilon(\cdot,\zeta_i^n)}_{\HH^{\alpha,p}}
=:I_g +R.
\end{multline*}
\noindent
We bound the term $I_g$ first: by \eqref{eq:heat_dual_representation}, we get
\begin{align*}
&\norm{\int_0^t\int_{\partial D}  P_{t-s} p_\epsilon(\cdot,y) g_s(y) dsd\sigma(y)}_{\HH^{\alpha,p}}\\ 
&\qquad=\sup_{\norm{f}_{L^{p'}}=1} \int_0^t\int_{\partial D}  \langle (I+A_p)^{\alpha/2}P_{t-s} p_\epsilon(\cdot,y) , f\rangle g_s(y) dsd\sigma(y)\\
&\qquad = \sup_{\norm{f}_{L^{p'}}=1}\int_0^t (I+A_p')^{\alpha/2}P_{t-s+\epsilon}f(y)g_s(y)dsd\sigma(y)\\
&\qquad \leq \|g\|_{L^\infty([0,1];L^p(\partial D))} \int_0^t \norm{(I+A_{p'})^{\alpha/2}P_{t-s+\epsilon}}_{L^{p'}\to L^{p'}(\partial D)} ds.
\end{align*}
Using the trace theorem and the contractivity bound \eqref{eq:contractivity}, we get, for some fixed $\delta>0$ such that $\alpha/2+1/(2p')+\delta/2<1$,
\begin{align}\nonumber
I_g&\le \|g\|_{L^\infty([0,1];L^p(\partial D))} \int_0^t \norm{(I+A_{p'})^{\alpha/2}P_{t-s+\epsilon}}_{L^{p'}\to H^{1/p'+\delta,p'}(D)} ds \\
\nonumber
&\lesssim \|g\|_{L^\infty([0,1];L^p(\partial D))} \int_0^t \norm{(I+A_{p'})^{\alpha/2+1/(2p')+\delta/2}P_{t-s+\epsilon}}_{L^{p'}\to L^{p'}} ds\\
\nonumber
&\lesssim_{\alpha,p} \|g\|_{L^\infty([0,1];L^p(\partial D))} \int_0^t (t-s+\epsilon)^{-(\alpha/2+1/(2p')+\delta/2)}  ds \\
\label{eq:bd_Ig}
&\lesssim_{\alpha,p} \|g\|_{L^\infty([0,1];L^p(\partial D))}.
\end{align}
Concerning $R$, recall that $A^0_n(t)=\{i: 0<t^n_i \le t\}$ and, for $i\in A_0^n(t)$, $\omega_i^n= \int_{Q^n_i} g_s(y)dsd\sigma(y)$. In particular, we can split
\begin{align*}
\sum_{i\in A^n_0(t)}\omega_i^n P_{t-t_i^n} p_\epsilon(\cdot,\zeta_i^n)&= \sum_{i: t_i^n+1/(2n)< t}\int_{Q^n_i} P_{t-t_i^n} p_\epsilon(\cdot,\zeta_i^n) g_s(y) dsd\sigma(y)\\
&\quad + \sum_{i:t_i^n\le t\le t_i^n+1/(2n)}\omega_i^n P_{t-t_i^n} p_\epsilon(\cdot,\zeta_i^n)
\end{align*}
(the last sum possibly being over an empty set) and similarly
\begin{align*}
&\int_0^t\int_{\partial D} P_{t-s} p_\epsilon(\cdot,y) g_s(y) dsd\sigma(y)\\
&= \sum_{i: t_i^n+1/(2n)< t}\int_{Q^n_i} P_{t-s} p_\epsilon(\cdot,y) g_s(y) dsd\sigma(y)\\
&\quad + \sum_{i:t_i^n-1/(2n)< t\le t_i^n+1/(2n)}\int_{Q^n_i\cap [0,t]\times \partial D} P_{t-s} p_\epsilon(\cdot,y) g_s(y) dsd\sigma(y).
\end{align*}
Hence we can split $R$ as
\begin{align*}
R&\le \norm{\sum_{i: t_i^n+1/(2n)< t}\int_{Q^n_i} \pa{P_{t-s} p_\epsilon(\cdot,y) - P_{t-t_i^n} p_\epsilon(\cdot,\zeta_i^n)} g_s(y) dsd\sigma(y)}_{\HH^{\alpha,p}}\\
&\quad +\norm{\sum_{i:t_i^n-1/(2n)< t\le t_i^n+1/(2n)}\int_{Q^n_i\cap [0,t]\times \partial D} P_{t-s} p_\epsilon(\cdot,y) g_s(y) dsd\sigma(y)}_{\HH^{\alpha,p}}\\
&\quad +\norm{\sum_{i:t_i^n\le t\le t_i^n+1/(2n)}\omega_i^n P_{t-t_i^n} p_\epsilon(\cdot,\zeta_i^n)}_{\HH^{\alpha,p}}\\
&=: R_1 +R_{21} +R_{22}.
\end{align*}
To bound the term $R_1$, we start with an observation: setting
\begin{equation*}
    [t^n_i,s](\xi):=t_i^n+\xi(s-t_i^n),\quad 
    [\zeta^n_i,y](\xi):=\zeta_i^n+\xi(y-\zeta_i^n),\quad
    \xi\in [0,1]
\end{equation*}
(notice that $[\zeta^n_i,y](\xi)\in \bar D$ for all $\xi$ since $\bar D$ is convex),
 since by definition of heat semigroup it holds $\partial_t P_t = -A_p P_t$,
 we can write
\begin{align*}
&P_{t-s} p_\epsilon(\cdot,y)(x) - P_{t-t_i^n} p_\epsilon(\cdot,\zeta_i^n)(x)\\
&= (s-t_i^n) \int_0^1 -\partial_t P_{t-[t^n_i,s](\xi)} p_\epsilon(\cdot,[\zeta^n_i,y](\xi))(x) d\xi\\
&\quad +(y-\zeta_i^n)\cdot  \int_0^1 \nabla_y P_{t-[t^n_i,s](\xi)} p_\epsilon(\cdot,[\zeta^n_i,y](\xi))(x) d\xi\\
&= (s-t_i^n) \int_0^1 A_p P_{t-[t^n_i,s](\xi)} p_\epsilon(\cdot,[\zeta^n_i,y](\xi))(x) d\xi\\
&\quad +(y-\zeta_i^n)\cdot  \int_0^1 \nabla_y P_{t-[t^n_i,s](\xi)} p_\epsilon(\cdot,[\zeta^n_i,y](\xi))(x) d\xi.
\end{align*}
Applying the latter to $R_1$ we obtain:
\begin{align*}
R_1&\le \norm{\sum_{i: t_i^n+\frac1{2n}< t}\int_{Q^n_i} \int_0^1 A_p P_{t-[t^n_i,s](\xi)} p_\epsilon(\cdot,[\zeta^n_i,y](\xi)) d\xi(s-t_i^n) g_s(y) dsd\sigma(y)}_{\HH^{\alpha,p}}\\
&\quad +\norm{\sum_{i: t_i^n+\frac1{2n}< t}\int_{Q^n_i} \int_0^1 \nabla_y P_{t-[t^n_i,s](\xi)} p_\epsilon(\cdot,[\zeta^n_i,y](\xi)) d\xi (y-\zeta_i^n) g_s(y) dsd\sigma(y)}_{\HH^{\alpha,p}}\\
&=:R_{11}+R_{12}.
\end{align*}
Concerning $R_{11}$, recall that $|s-t_i^n|\le 1/(2n)$ for every $s$ in $Q_i^n$. Hence by \eqref{eq:heat_norm_sup} we have
\begin{align*}
R_{11}&\le \sum_{i: t_i^n+\frac1{2n}< t}\int_{Q^n_i} \int_0^1 \norm{A_p P_{t-[t^n_i,s](\xi)} p_\epsilon(\cdot,[\zeta^n_i,y](\xi))}_{\HH^{\alpha,p}} d\xi \, |s-t_i^n| g_s(y) dsd\sigma(y)\\
&\le \frac{1}{2n} \int_0^1 \sum_{i: t_i^n+\frac1{2n}< t}\int_{Q_i^n} \sup_{x\in \bar{D}}\norm{P_{t-[t^n_i,s](\xi)}p_\epsilon(\cdot,x)}_{\HH^{\alpha+2,p}} g_s(y) dsd\sigma(y)d\xi,\\
&= \frac{1}{2n} \int_0^1 \sum_{i: t_i^n+\frac1{2n}< t}\int_{Q_i^n} \norm{(I+A_{p'})^{1+\alpha/2}P_{t-[t^n_i,s](\xi)+\epsilon}}_{L^{p'}\to C(\bar{D})} g_s(y) dsd\sigma(y)d\xi.
\end{align*}
Since $\HH^{2,p'}$ is embedded into $C(\bar{D})$ for $p'>1$, we have
\begin{align*}
R_{11}&\lesssim_p \frac{1}{n} \int_0^1 \sum_{i: t_i^n+\frac1{2n}< t}\int_{Q_i^n} \norm{(I+A_{p'})^{2+\alpha/2}P_{t-[t^n_i,s](\xi)+\epsilon}}_{L^{p'}\to L^{p'}} g_s(y) dsd\sigma(y)d\xi.
\end{align*}
Note that, if $t_i^n+1/(2n)< t$, then $t-[t^n_i,s](\xi)\ge (t-s)/2$ for every $s$ in $Q_i^n$. Thanks to the contractivity bound \eqref{eq:contractivity}, we get
\begin{align}
R_{11}+
&\lesssim_{\alpha,p} \frac{1}{n} \int_0^1 \sum_{i: t_i^n+\frac1{2n}< t}\int_{Q_i^n} (t-[t^n_i,s](\xi)+\epsilon)^{-2-\alpha/2} g_s(y) dsd\sigma(y)d\xi \nonumber\\
&\lesssim \frac{1}{n} \int_0^t\int_{\partial D} (t-s+2\epsilon)^{-2-\alpha/2} g_s(y) d\sigma(y)ds \nonumber\\
&\lesssim \frac{1}{n} \epsilon^{-1-\alpha/2} \|g\|_{L^\infty([0,1];L^1(\partial D))}.\label{eq:bd_R11}
\end{align}
Concerning $R_{12}$, recall that $|y-\zeta_i^n|\lesssim 1/n$ for every $y$ in $Q_i^n$. Hence by \eqref{eq:heat_norm_grad_sup} we have
\begin{align*}
R_{12}&\le \sum_{i: t_i^n+\frac1{2n}< t}\int_{Q^n_i} \int_0^1 \norm{\nabla_y P_{t-[t^n_i,s](\xi)} p_\epsilon(\cdot,[\zeta^n_i,y](\xi))}_{\HH^{\alpha,p}} d\xi \, |y-\zeta_i^n| g_s(y) dsd\sigma(y)\\
&\lesssim \frac{1}{n} \int_0^1 \sum_{i: t_i^n+\frac1{2n}< t}\int_{Q_i^n} \sup_{x\in \bar{D}}\norm{\nabla P_{t-[t^n_i,s](\xi)}p_\epsilon(\cdot,x)}_{\HH^{\alpha,p}} g_s(y) dsd\sigma(y)d\xi\\
&\lesssim \frac{1}{n} \int_0^1 \sum_{i: t_i^n+\frac1{2n}< t}\int_{Q_i^n} \norm{\nabla P_{t-[t^n_i,s](\xi) +\epsilon}(I+A_{p'})^{\alpha/2}}_{L^{p'}\to C(\bar{D})} g_s(y) dsd\sigma(y)d\xi.
\end{align*}
Thanks to the gradient bound Lemma \ref{lem:gradientdistributions}, we get
\begin{align}
R_{12}&\lesssim_{\alpha,p} \frac{1}{n} \int_0^1 \sum_{i: t_i^n+\frac1{2n}< t}\int_{Q_i^n} (t-[t^n_i,s](\xi)+\epsilon)^{-1/2-\alpha/2-1/p'} g_s(y) dsd\sigma(y)d\xi \nonumber\\
&\lesssim \frac{1}{n} \int_0^t\int_{\partial D} (t-s+2\epsilon)^{-3/2-\alpha/2} g_s(y) d\sigma(y)ds \nonumber\\
&\lesssim_{\alpha,p} \frac{1}{n} \epsilon^{-1/2-\alpha/2} \|g\|_{L^\infty([0,1];L^1(\partial D))}.\label{eq:bd_R12}
\end{align}
We turn now to the bound on $R_{21}$. By \eqref{eq:heat_norm_grad_sup} we have
\begin{align*}
R_{21}&\le \sum_{i:t_i^n-\frac1{2n}< t\le t_i^n+\frac1{2n}}\int_{Q^n_i\cap [0,t]\times \partial D} \norm{P_{t-s} p_\epsilon(\cdot,y)}_{\HH^{\alpha,p}}\, g_s(y) dsd\sigma(y)\\
&\le \sum_{i:t_i^n-\frac1{2n}< t\le t_i^n+\frac1{2n}}\int_{Q^n_i\cap [0,t]\times \partial D} \sup_{x\in \bar{D}}\norm{P_{t-s+\epsilon} p_\epsilon(\cdot,x)}_{\HH^{\alpha,p}} \, g_s(y) dsd\sigma(y)\\
&= \sum_{i:t_i^n-\frac1{2n}< t\le t_i^n+\frac1{2n}}\int_{Q^n_i\cap [0,t]\times \partial D} \norm{(I+A_{p'})^{\alpha/2}P_{t-s+\epsilon}}_{L^{p'}\to C(\bar{D})} \, g_s(y) dsd\sigma(y).
\end{align*}
Notice that
\begin{align*}
\bigcup_{i:t_i^n-\frac1{2n}< t\le t_i^n+\frac1{2n}}Q_i^n \cap ([0,t]\otimes \partial D) \subseteq [t-1/n,t]\times\partial D.
\end{align*}
Therefore, by the embedding $\HH^{2,p'}\hookrightarrow C(\bar{D})$, $p'>1$, and contractivity bound \eqref{eq:contractivity},
\begin{align}
R_{21}&\lesssim_{\alpha,p} \int_{t-1/n}^t\int_{\partial D} g_s(y)d\sigma(y)ds \cdot \sup_{s\le t}\norm{(I+A_{p'})^{1+\alpha/2}P_{t-s+\epsilon}}_{L^{p'}\to L^{p'}} \nonumber\\
&\lesssim \frac{1}{n}\epsilon^{-1-\alpha/2}\|g\|_{L^\infty([0,1];L^1(\partial D)}.\label{eq:bd_R21}
\end{align}
Finally, we turn to the bound on $R_{22}$. Similarly to the case of $R_{21}$, we notice that for each $i\in A_0(t)$,
\begin{align*}
\cup_{i:t_i^n\le t\le t_i^n+\frac1{2n}}Q_i^n \subseteq \bra{t-1/n,t+\frac1{2n}}\times\partial D.
\end{align*}
Hence, proceeding as for the term $R_{21}$, we get
\begin{align}
R_{22}&\le \sum_{i:t_i^n\le t\le t_i^n+\frac1{2n}}\int_{Q^n_i}g_s(y) dsd\sigma(y) \,\norm{P_{t-t_i^n} p_\epsilon(\cdot,\zeta_i^n)}_{\HH^{\alpha,p}} \nonumber\\
&\le \int_{t-1/n}^{t+1/(2n)}\int_{\partial D} g_s(y)d\sigma(y)ds \, \sup_{s\le t}\norm{(I+A_{p'})^{1+\alpha/2}P_{t-s+\epsilon}}_{L^{p'}\to L^{p'}} \nonumber\\
&\lesssim \frac{1}{n}\epsilon^{-1-\alpha/2}\|g\|_{L^\infty([0,1];L^1(\partial D)}.\label{eq:bd_R22}
\end{align}
Putting together the bounds \eqref{eq:bd_Ig}, \eqref{eq:bd_R11}, \eqref{eq:bd_R12}, \eqref{eq:bd_R21}, \eqref{eq:bd_R22}, we conclude that
\begin{multline*}
\norm{\sum_{i\in A^n_0(t)}\omega_i^n P_{t-t_i^n} p_\epsilon(\cdot,\zeta_i^n)}_{\HH^{\alpha,p}}\\
\lesssim_{\alpha,p} \|g\|_{L^\infty([0,1];L^p(\partial D))} +\frac{1}{n}\epsilon^{-1-\alpha/2}\|g\|_{L^\infty([0,1];L^1(\partial D))}.
\end{multline*}
By the assumption $\epsilon=\epsilon(n)\gtrsim n^{-2/(2+\alpha)}$, we get the desired bound.
\end{proof}

We also estimate in a dedicated Lemma the martingale terms appearing
in the mild formulation of particle dynamics.

\begin{lemma}\label{lem:martingaleterm}
    Let $p>2$, $\frac{2}{p}<\alpha<1$ and assume that 
    $\epsilon=\epsilon(n)\gtrsim n^{-1/2}$.
    Then,
    \begin{equation}\label{eq:martingaletermbound}
        \expt{\sup_{t\in[0,1]}\norm{\int_{0}^t \sum_{i\in A^n(s)}\omega_i^n P_{t-s} \nabla_y p_\epsilon(x,x_i^n(s))dB^i_s}^q_{\HH^{\alpha,p}}}
        \leq C_{q,p,\nu,\alpha}.
    \end{equation}
\end{lemma}

\begin{proof}
    By Sobolev embedding $\HH^{1-2/p,2}\hookrightarrow L^p$
    and Burkholder--Davis--Gundy inequality \eqref{eq:bdgconv}
    it holds
    \begin{multline*}
        \expt{\sup_{t\in[0,1]}\norm{\int_0^t \sum_{i\in A^n(s)}\omega_i^n(I+A_p)^{\alpha/2}P_{t-s} \nabla_y p_\epsilon(\cdot,x_i(s))dB^i_s}_{L^p}^q}\\
        \lesssim_p \expt{\sup_{t\in[0,1]}\norm{\int_0^t \sum_{i\in A^n(s)}\omega_i^n (I+A_2)^{\left(1+\alpha-2/p\right)/2} P_{t-s} \nabla_y p_\epsilon(\cdot,x_i(s))dB^i_s}_{L^2}^q}\\
        \lesssim_q \expt{\pa{\int_0^1 \sum_{i\in A^n(s)}(\omega_i^n)^2
        \norm{(I+A_2)^{\left(1+\alpha-2/p\right)/2} \nabla_y p_\epsilon(\cdot,x_i(s))}^2_{L^2}ds}^{q/2}}.
    \end{multline*}
    The inner integrand in the right-hand side
    can be controlled uniformly with respect to the
    particles' positions: by \autoref{lem:gradientdistributions}
    \begin{multline*}
        \sup_{y\in \bar D} \norm{(I+A_2)^{\left(1+\alpha-2/p\right)/2} \nabla_y p_\epsilon(\cdot,y)}^2_{L^2}\\
        =\norm{\nabla P_\epsilon}_{\HH^{-1-\alpha+2/p,2}\to C(\bar D)}^2
        \lesssim_{\alpha,\nu,p} \epsilon^{-3-\alpha+2/p}.
    \end{multline*}
    Notice that the denominator is integrable in $ds$ for all $\delta>0$.
    We now take into account the fact that $\omega^n_i\lesssim n^{-2}(\|\omega_0\|_{L^\infty(D)}+\|g\|_{L^\infty([0,1]\times\partial D})$
    uniformly in $i$, and that $N(n)\simeq n^2$. Hence, the left-hand side of \eqref{eq:martingaletermbound} is bounded
    from above by
    \begin{equation*}
        \text{L.H.S.}\lesssim_{p,q,\nu,\alpha} n^{-q} \epsilon^{-\pa{3+\alpha-2/p}\frac{q}{2}}(\|\omega_0\|_{L^\infty(D)}+\|g\|_{L^\infty([0,1]\times\partial D})^q.
    \end{equation*}
    Taking $\epsilon(n)$ as in the hypothesis, the last quantity is uniformly bounded in $n$.
\end{proof}

\subsection{Estimates for Approximating Solutions}
The following is the core estimate of our argument, 
providing uniform boundedness in strong norms for the approximating processes.

\begin{proposition}\label{prop:unifspacebound}
	Let $p>2$, $\frac{2}{p}<\alpha<1$, and $\epsilon=\epsilon(n)\gtrsim n^{-1/2}$.
    For all $q\geq 2$ it holds
	\begin{equation}\label{eq:spacebound}
		\sup_{n}\expt{\sup_{t\in [0,1]}\norm{\omega^n_t}^q_{\HH^{\alpha,p}}}\le C_{M,q,p,\nu,\alpha}.
	\end{equation}
\end{proposition}

\begin{proof}
    In order to lighten notation we denote by $v_t^n(x):=F\pa{\int_D K_n(x,y)dS^n_s(y)}$
    the vector field acting on a single, smoothened particle.
    The mild formulation \eqref{eq:duhamelformula} allows to bound,
    by Minkowski inequality,
    \begin{align*}
        \expt{\norm{\omega^n_t}^q_{\HH^{\alpha,p}}}
        &\lesssim_q \expt{\norm{\sum_{i\in A^n(t)}\omega_i^n P_{t-t_i^n} p_\epsilon(\cdot,\zeta_i^n)}^q_{\HH^{\alpha,p}}}\\
        &\quad +\expt{\norm{\int_0^t P_{t-s} \int_D v^n_s(y)\nabla_y p_\epsilon(x,y)dS^n_s(y)ds}^q_{\HH^{\alpha,p}}}\\
        &\quad +\expt{\norm{\sqrt{2\nu} \int_{0}^t \sum_{i\in A^n(s)}\omega_i^n P_{t-s} \nabla_y p_\epsilon(x,x_i^n(s))dB^i_s}^q_{\HH^{\alpha,p}}}\\
        &=: I_1+I_2+I_3.
    \end{align*}
    The initial and boundary creation terms $I_1$ and the martingale term $I_3$ are uniformly bounded (both in $n$ and in $t\in [0,1]$) respectively 
    thanks to \autoref{lem:initialdatum}, \autoref{lem:boundarygeneration}
    and \autoref{lem:martingaleterm}.
    We thus focus on the nonlinear interaction term $I_2$,
    where the regularizing effect of the heat semigroup
    is essential. Using the gradient estimate \eqref{eq:gradientestimateBASE}
    we thus bound, for $f\in L^{p'}$,
   \begin{align*}
       &\abs{\sum_{i\in A^n(s)}\omega_i v^n_s(x_i(s))\int_D f(x)(I+A_p)^{\alpha/2}P_{t-s}
       \nabla_y p_\epsilon(x,x_i(s))dx}\\
       &\quad =\abs{\sum_{i\in A^n(s)}\omega_i v^n_s(x_i(s))\int_D (I+A_{p'})^{\alpha/2}P_{t-s}f(x)\nabla_y p_\epsilon(x,x_i(s))dx}\\
       &\quad =\abs{ \sum_{i\in A^n(s)}\omega_i v^n_s(x_i(s)) \nabla P_\epsilon[(I+A_{p'})^{\alpha/2}P_{t-s}f](x_i(s))}\\
       &\quad \leq M\sum_{i\in A^n(s)}\omega_i  \abs{\nabla P_\epsilon[(I+A_{p'})^{\alpha/2}P_{t-s}f](x_i(s))}\\
       &\quad \leq M\sum_{i\in A^n(s)}\omega_i e^{C\epsilon}P_\epsilon\abs{\nabla(I+A_{p'})^{\alpha/2}P_{t-s}f}(x_i(s))\\
       &\quad =Me^{C\epsilon} \sum_{i\in A^n(s)}\omega_i  \int_D p_\epsilon(x,x_i(s))  \abs{\nabla(I+A_{p'})^{\alpha/2}P_{t-s}f}(x)dx\\
       &\quad = Me^{C\epsilon} \int_D P_\epsilon(S^n_s)(x)
      \abs{\nabla(I+A_{p'})^{\alpha/2}P_{t-s}f}(x)dx\\
      &\quad \leq Me^{C\epsilon} \norm{\omega^n_s}_{L^p} \norm{\nabla(I+A_{p'})^{\alpha/2}P_{t-s}f}_{L^{p'}}.
   \end{align*}
   By duality, this implies that
   \begin{multline*}
       \norm{\sum_{i\in A^n(s)}\omega_i v^n_s(x_i(s))(I+A_p)^{\alpha/2}P_{t-s}\nabla_y p_\epsilon(\cdot,x_i(s))}_{L^p}\\
       \leq
       \frac{Me^{C\epsilon}}{(t-s)^{(1+\alpha)/2}} \norm{\omega^n_s}_{L^p}
       \leq \frac{Me^{C\epsilon}}{(t-s)^{(1+\alpha)/2}} \norm{\omega^n_s}_{\HH^{\alpha,p}},
   \end{multline*}
   so that
   \begin{multline*}
       I_2^{1/q}\leq \int_0^t \expt{\norm{\sum_{i\in A^n(s)}\omega_i v^n_s(x_i(s))P_{t-s}\nabla_y p_\epsilon(\cdot,x_i(s))}_{\HH^{\alpha,p}}^q}^{1/q}ds\\
       \leq C_{\epsilon}\int_0^t \frac{M}{(t-s)^{(1+\alpha)/2}}\mathbb{E}\left[\lVert \omega^n_s\rVert_{\HH^{\alpha,p}}^q\right]^{1/q}\ ds
   \end{multline*}
   with the constant $C_\epsilon$ decreasing as $\epsilon\rightarrow 0$,
   in particular uniformly bounded in $n$.
   Altogether, we arrive to the integral inequality
   \begin{equation*}
    \mathbb{E}\left[\norm{\omega^n_s}^q_{\HH^{\alpha,p}}\right]^{1/q}\leq C_{p,q,\nu,\alpha,\omega_0}+ C\int_0^t \frac{M}{(t-s)^{(1+\alpha)/2}}\mathbb{E}\left[\lVert \omega^n_s\rVert_{\HH^{\alpha,p}}^q\right]^{1/q}\ ds
   \end{equation*}
   which, by Gr\"onwall's lemma (as $\alpha<1$), implies
   \begin{equation}\label{eq:partialbound}
       \sup_{n}\sup_{t\in [0,1]}\expt{\norm{\omega^n_t}^q_{\HH^{\alpha,p}}}\le C_{M,q,p,\nu,\alpha}.
   \end{equation}
   In order to complete the proof, we apply again the mild formulation
   of the dynamics \eqref{eq:duhamelformula} to estimate
   \begin{align*}
        \expt{\sup_{t\in [0,1]}\norm{\omega^n_t}^q_{\HH^{\alpha,p}}}
        &\lesssim_q \expt{\sup_{t\in [0,1]}\norm{\sum_{i\in A^n(t)}\omega_i^n P_{t-t_i^n} p_\epsilon(\cdot,\zeta_i^n)}^q_{\HH^{\alpha,p}}}\\
        &\quad +\expt{\sup_{t\in [0,1]}\norm{\int_0^t P_{t-s} \int_D v^n_s(y)\nabla_y p_\epsilon(x,y)dS^n_s(y)ds}^q_{\HH^{\alpha,p}}}\\
        &\quad +\expt{\sup_{t\in [0,1]}\norm{\sqrt{2\nu} \int_{0}^t \sum_{i\in A^n(s)}\omega_i^n P_{t-s} \nabla_y p_\epsilon(x,x_i^n(s))dB^i_s}^q_{\HH^{\alpha,p}}}\\
        &=: J_1+J_2+J_3.
    \end{align*}
    Once again, the initial and boundary creation terms $J_1$ and the martingale term $J_3$ are uniformly bounded in $n$ respectively 
    thanks to \autoref{lem:initialdatum}, \autoref{lem:boundarygeneration}
    and \autoref{lem:martingaleterm}.
    As for $J_2$, since
    \begin{equation*}
       J_2\leq \expt{\pa{\sup_{t\in [0,1]}\int_0^t \norm{P_{t-s} \int_D v^n_s(y)\nabla_y p_\epsilon(x,y)dS^n_s(y)}_{\HH^{\alpha,p}}ds}^q},
   \end{equation*}
   we can repeat the computations performed to control $I_2$ obtaining
   \begin{equation*}
       J_2\lesssim_{p,q,\nu,\alpha} 
       \expt{\sup_{t\in [0,1]}\pa{\int_0^t \frac{Me^{C\epsilon}}{(t-s)^{(1+\alpha)/2}} \norm{\omega^n_s}_{\HH^{\alpha,p}}ds}^q},
   \end{equation*}
   where we now apply H\"older inequality,
   \begin{equation*}
       \int_0^t \norm{\omega^n_s}_{\HH^{\alpha,p}} \frac{ds}{(t-s)^{(1+\alpha)/2}}
       \leq \pa{\int_0^t \frac{ds}{(t-s)^{r(1+\alpha)/2}}}^{1/r}
       \pa{\int_0^t \norm{\omega^n_s}_{\HH^{\alpha,p}}^{r'}ds }^{1/r'},
   \end{equation*}
   in which we can choose a small enough $r>0$ (depending on $\alpha$)
   so that the first factor on the right-hand side is finite and $q/r'<1$, hence
   \begin{equation*}
       J_2\lesssim_{p,q,\nu,\alpha} 
       \expt{\int_0^1 \norm{\omega^n_s}_{\HH^{\alpha,p}}^{r'}ds}
       \leq \sup_{t\in [0,1]} \expt{\norm{\omega^n_t}_{\HH^{\alpha,p}}^{r'}},
   \end{equation*}
   which is uniformly bounded in $n$ by \eqref{eq:partialbound}.
\end{proof}

It is worth noticing that the gradient estimates for $P_t$ are crucial
in controlling the nonlinear term of the dynamics, but the regularizing effect
of $P_t$ was neglected in estimating the initial empirical measure in \autoref{lem:initialdatum} and the martingale terms in \autoref{lem:martingaleterm} (due to the application of BDG inequality).
The singularity of $\nabla p_\epsilon$ appearing in the initial empirical measure and in the stochastic integrals 
is \emph{not} improved by the heat semigroup and produces a restriction on the asymptotic behavior of $\epsilon=\epsilon(n)$.

Thanks to the good control provided by \autoref{prop:unifspacebound},
we can estimate time increments 
--which we need for the  equicontinuity part of our compactness argument--
in a much weaker norm, thus allowing to exploit the weak formulation of the dynamics,
easier to deal with compared to the variation-of-constants form but producing bounds
in weaker norms.

\begin{proposition}\label{prop:uniftimebound}
    Let $p>2$, $\frac{2}{p}<\alpha<1$, $\epsilon=\epsilon(n)\gtrsim n^{-1/2}$,
    $q\geq 2$, and $(\tau_n)_{n\in \N_0}$ be a sequence of $\F$-stopping times in $[0,1]$.
    It holds, for all $r\in (0,1)$,
    \begin{gather} \label{eq:timebound}
		\expt{\norm{\omega^n_{(\tau_n+r)\wedge 1}-\omega^n_{\tau_n}}^q_{\HH^{-2,p}}}
		\leq \left(r^{q/2}+\frac{1}{n^q}\right)C_{M,q,p,\nu,\alpha}.
	\end{gather}
\end{proposition}

\begin{proof}
Substituting the weak formulation of particle dynamics \eqref{eq:itoOmega},
we estimate by Minkowski inequality:
\begin{align*}
    &\expt{\norm{\omega^n_{(\tau_n+r)\wedge 1}-\omega^n_{\tau_n}}^q_{\HH^{-2,p}}}\\
            &\qquad\lesssim_q \expt{\norm{\sum_{i\in A^n_0((\tau_n+r)\wedge 1)\setminus A^n_0(\tau_n)}\omega_i^n p_\epsilon(\cdot,\zeta_i^n)}_{\HH^{-2,p}}^q}\\ 
            &\qquad+\expt{\norm{\int_{\tau_n}^{(\tau_n+r)\wedge 1} \nu\Delta \omega^n_sds}_{\HH^{-2,p}}^q}\\ \nonumber
			&\qquad +\expt{\norm{\int_{\tau_n}^{(\tau_n+r)\wedge 1} \int_D F\pa{K[\omega^n](y)}\nabla_y p_\epsilon(\cdot,y)dS^n_s(y)ds}_{\HH^{-2,p}}^q}\\ \nonumber
			&\qquad  +\expt{\norm{\int_{\tau_n}^{\tau_{n}+r} \sqrt{2\nu} \sum_{i\in A^n(s)}\omega_i^n \nabla_y p_\epsilon(\cdot,x_i^n(s))dB^i_s}_{\HH^{-2,p}}^q}\\ 
            &\qquad =:I_1+I_2+I_3+I_4.
\end{align*}
We now proceed estimating each term separately.

\emph{Estimates on generation term} $I_1$.
\begin{align*}
    I_1
    & \leq 
    \expt{\pa{\sum_{i\in A^n(1)}\omega_i^n\one_{\tau_n<t_i\leq (\tau_n+r)\wedge 1}\norm{(I+A_p)^{-1}p_{\epsilon}(\cdot,\zeta_i^n)}_{L^p}}^q}\\ 
    & \leq 
    \expt{\pa{\sum_{i\in A^n(1)}\omega_i^n\one_{\tau_n<t_i\leq (\tau_n+r)\wedge 1}}^q}\sup_{y\in \partial D}\norm{(I+A_p)^{-1}p_{\epsilon}(\cdot,y)}_{L^p}^q.
\end{align*}
The first factor on the right-hand side of the latter is controlled by
\begin{multline*}
    \expt{\pa{\sum_{i\in A^n(1)}\omega_i^n\one_{\tau_n<t_i\leq (\tau_n+r)\wedge 1}}^q}\\
    \leq \expt{\left(\int_{\tau_n-\frac{1}{n}\lor 0}^{\tau_n+\frac{1}{n}+r\wedge 1}\int_{\partial D}g(s,y) dyds\right)^q }
    \leq \left(r+\frac{2}{n}\right)^q\|g\|_{L^\infty([0,1];L^1(\partial D))}^q,
\end{multline*}
so we are left to control the second factor.
For $f\in L^{p'}$, $y\in \de D$ and a sequence $y_k\in D$ converging to $y$,
it holds
\begin{align*}
    \int_D (I+A_p)^{-1}p_{\epsilon}(x,y)f(x) dx
    &=\int_D p_{\epsilon}(x,y)(I+A_{p'})^{-1}f(x) dx\\
    &=\lim_{k\to\infty}\int_D p_{\epsilon}(x,y_k)(I+A_{p'})^{-1}f(x) dx,
\end{align*}
the exchange between the limit and the integral being allowed thanks to the fact that, by \autoref{prop:gradientestimates}, $p_{\epsilon}(x,y)\leq \frac{C}{\epsilon}e^{-\lvert x-y\rvert^2/(c\epsilon)}$ $\forall x,y\in D$.
We can thus estimate:
\begin{align*}
    &\sup_{y\in \partial D}\norm{(I+A_p)^{-1}p_{\epsilon}(\cdot,y)}_{L^p}
    =\sup_{{y\in \partial D}, \norm{f}_{L^{p'}}=1}
    \abs{\int_D (I+A_p)^{-1}p_{\epsilon}(x,y)f(x) dx}\\ 
    &\qquad\leq \sup_{{y\in \partial D}, \norm{f}_{L^{p'}}=1}\lim_{k\to\infty}\left\lvert\int_D p_{\epsilon}(x,y_k)(I+A_{p'})^{-1}f(x) dx\right\rvert\\ 
    &\qquad\leq  \sup_{{y\in  \bar{D}}, \norm{f}_{L^{p'}}=1}\lvert (I+A_{p'})^{-1}P_{\epsilon}f(y)\rvert
    \leq \norm{(I+A_{p'})^{-1}P_{\epsilon}}_{L^{p'}\to C(\bar{D})}\\ 
    &\qquad\leq \norm{(I+A_{p'})^{-1}}_{L^{p'}\to C(\bar{D})}
    \norm{P_{\epsilon}}_{L^{p'}\to L^{p'}}\lesssim_{p,\nu} 1,
\end{align*}
where $\norm{(I+A_{p'})^{-1}}_{L^{p'}\to C(\bar{D})}\lesssim_p\norm{(I+A_{p'})^{-1}}_{L^{p'}\to \HH^{2,p'}}\lesssim_p 1$
follows by Morrey's inequality.

\emph{Estimates on the diffusion term} $I_2$ follow from a 
straightforward application of H\"older inequality,
\begin{align*}
    I_2 & \leq\expt{\pa{\int_{\tau_n}^{(\tau_n+r)\wedge 1}\norm{\nu\Delta\omega^n_s}_{\HH^{-2,p}}ds}^q}\\
    &\lesssim_{\nu}r^{q-1} \expt{\int_{\tau_n}^{(\tau_n+r)\wedge 1}\norm{\omega^n_s}^q_{L^p}ds }
    \leq r^{q}\expt{\sup_{t\in [0,1]}\norm{\omega^n_t}_{L^p}^q}.
\end{align*}

\emph{Estimates on the nonlinear term} $I_3$.
We have:
\begin{align*}
    I_3& \leq \expt{\pa{\int_{\tau_n}^{(\tau_n+r)\wedge 1}\norm{\int_D F\pa{K[\omega^n](y)}\nabla_y p_\epsilon(\cdot,y)dS^n_s(y)}_{\HH^{-2,p}}ds}^q}\\ & \leq r^{q-1}\expt{\int_{\tau_n}^{(\tau_n+r)\wedge 1}\norm{\int_D F\pa{K[\omega^n](y)}\nabla_y p_\epsilon(\cdot,y)dS^n_s(y)}_{\HH^{-2,p}}^qds},
\end{align*}
in which we control, somewhat analogously to \autoref{prop:unifspacebound},
\begin{align*}
    &\norm{\int_D F\pa{K[\omega^n](y)}\nabla_y p_\epsilon(\cdot,y)dS^n_s(y)}_{\HH^{-2,p}}\\ 
    &\quad =\norm{\int_D  F\pa{K[\omega^n](y)}(I+A_p)^{-1}\nabla_y p_\epsilon(\cdot,y)dS^n_s(y)}_{L^p}\\ 
    &\quad =\sup_{f\in L^{p'},\norm{f}_{L^{p'}}=1} \int_D F(K[\omega^n](y)) \langle (I+A_p)^{-1}\nabla_y p_{\epsilon}(\cdot,y),f\rangle dS^n_s(y)\\ 
    &\quad  \lesssim_M \sup_{f\in L^{p'},\norm{f}_{L^{p'}}=1} \int_D \lvert\nabla P_{\epsilon}(I+A_{p'})^{-1}f(y)\rvert dS^n_s(y)\\ 
    &\quad  \lesssim_{\nu,p} e^{C\epsilon}\sup_{f\in L^{p'},\norm{f}_{L^{p'}}=1}\int_D dS^n_s(y)\int_D p_{\epsilon}(y,x)\lvert \nabla(I+A_{p'})^{-1}f(x)\rvert dx\\ 
    &\quad =\sup_{f\in L^{p'},\norm{f}_{L^{p'}}=1} \brak{\omega^n_s,\abs{\nabla(I+A_{p'})^{-1}f}}
    \lesssim \norm{\omega^n_s}_{L^p}.
\end{align*}
Therefore $I_3$ is bounded, up to some constant independent of $n$, by $$r^{q}\expt{\sup_{t\in [0,1]}\norm{\omega^n_t}_{L^p}^q}.$$

\emph{Estimates on the martingale term} $I_4$.
By Sobolev embedding and Burkholder--Davis--Gundy inequality, \autoref{prop:bdg}, 
\begin{align*}
    I_4 
    & \lesssim_{\nu,p} \expt{\norm{\int_{\tau_n}^{\tau_{n}+r}  \sum_{i\in A^n(s)}\omega_i^n (I+A_2)^{-1}\nabla_y p_\epsilon(\cdot,x_i^n(s))dB^i_s}_{\HH^{1-2/p,2}}^q}\\ 
    & = \expt{\norm{\sum_{i=1}^N\omega_i^n \int_{0}^{1} \one_{t_i^n\geq s}\one_{\tau_n\leq s\leq (\tau_n+r)\wedge 1}(I+A_2)^{-1/2-1/p}\nabla_y p_\epsilon(\cdot,x_i^n(s))dB^i_s}_{L^2}^q}
    \\ 
    & \lesssim_q \expt{\pa{\sum_{i=1}^N(\omega^n_i)^2\int_{\tau_n}^{(\tau_n+r)\wedge 1}\one_{t_i^n\geq s}\norm{(I+A_2)^{-1/2-1/p}\nabla_y p_{\epsilon}(\cdot,x_i^n(s))}^2_{L^2}ds}^{\frac{q}2}}.
\end{align*}
Since $N=N(n)\simeq n^2$ and $\omega^n_i\lesssim n^{-2}\|g\|_{L^\infty([0,1]\times \partial D)}$ uniformly in $i$,
and replacing the expectation involving particles's positions
with a supremum over $\bar D$, we obtain
\begin{align*}
    I_4
    &\lesssim_{\nu,p,q} \frac{r^{q/2}}{n^q}
    \sup_{y\in \bar D}\norm{(I+A_2)^{-1/2-1/p}\nabla_y p_{\epsilon}(\cdot,y)}^q_{L^2},
\end{align*}
which we combine with a consequence of \autoref{lem:gradientdistributions},
\begin{equation*}
    \sup_{y\in \bar D}\norm{(I+A_2)^{-1/2-1/p}\nabla_y p_{\epsilon}(\cdot,y)}_{L^2}
    \leq \sup_{y\in \bar D}\lVert \nabla_y p_{\epsilon}(\cdot,y)\rVert_{L^2}
    =\lVert \nabla P_{\epsilon} \rVert_{L^{2}\rightarrow C(\bar D)}\lesssim \epsilon^{-1},
\end{equation*}
and the assumption $\epsilon(n)\gtrsim n^{-1/2}$, concluding
that $I_4\lesssim_{\nu,p,q} r^{q/2}$.
\end{proof}

\section{Proof of the Main Result}\label{sec:maintheorem}

The proof of \autoref{thm:main} proceeds as follows.
We combine the estimates of the previous section with Aldous'
Lemma in order to obtain tightness on the space of \emph{càdlàg}
functions taking values in Sobolev spaces.
In order to verify that the limiting dynamics coincides with the Navier-Stokes
equations \eqref{eq:neuNS} we need almost sure convergence in such a space,
which we obtain by changing the underlying probability space
by Skorohod theorem. Convergence in the original probability space is recovered by uniqueness of the deterministic limit.

In what follows, according to the hypothesis of \autoref{thm:main},
we tacitly assume that $p>2$ and $\frac{2}{p}<\alpha<1$ are fixed,
as well as $\omega_0\in \lip(\bar D) $ and $g\in \B_b([0,1]\times \de D)$.
We denote by $\omega$ the unique solution of \eqref{eq:neuNS} given
by \autoref{prop:H-1method} with initial datum $\omega_0$ and Neumann source $g$.

\subsection{Compactness Argument}\label{sec:compactness}
Let $(\Omega,\F,\PP)$ be a complete, filtered probability space satisfying the standard hypothesis, on which it is defined a sequence $(B^i)_{i\in \N}$ of independent $\F$-Brownian motions. 
For all $n\geq 1$, \emph{on this same} probability space
we can consider the well-posed dynamics of \autoref{prop:wellposSDE},
since the latter is a probabilistically strong existence and uniqueness result.
In other words, we can consider for all $n$ the dynamics of particles $(x_1^n(t),\dots, x_N^n(t))_{t\in [0,1]}$ and the one of the regularized empirical measure $\omega^n_t=\sum_{i\in A^n(t)}\omega_i^n p_\epsilon(\cdot,x_i^n(t))$
as stochastic processed defined on $(\Omega,\F,\PP)$ and $\F$-adapted.

We denote by $\L_n$ the law of $\omega_n$ on $\DD([0,1],\HH^{\alpha,p})$;
we can actually consider any parameters $\alpha,p$ since 
samples of the process $\omega^n$ are smooth in the space variable,
but time dependence is at best $\emph{càdlàg}$ due to the creation of new particles.

We need first to show tightness of the laws $(\L_n)_{n\in\N}$:

\begin{lemma}\label{lem:tightness}
    For all $\frac{2}{p}<\eta<\alpha$, the sequence of laws $(\L_n)_{n\in\N}$
    is tight on $\DD([0,T];\HH^{\eta,p})$.
\end{lemma}
\noindent
We then need to show that the limiting dynamics is actually 
a solution of \eqref{eq:neuNS}.

\begin{lemma}\label{lem:limitingdynamics}
    For all $\frac{2}{p}<\eta<\alpha$,
    any weak limit point of the sequence $(\L_n)_{n\in\N}$ of measures on 
    $\DD([0,T];\HH^{\eta,p})$
    is concentrated on the unique solution $\omega$ of \eqref{eq:neuNS}
    given by \autoref{prop:H-1method};
    in other words there exists a unique limit point, 
    $\delta_\omega$.
\end{lemma}
\noindent
These two Lemmas imply \autoref{thm:main}.

\begin{proof}[Proof of \autoref{thm:main}]
By \autoref{lem:tightness} and \autoref{lem:limitingdynamics}, every subsequence $\L_{n(k)}$ admits a sub-subsequence which converges to the unique limit point $\delta_\omega$, where $\omega$ is the deterministic solution of \eqref{eq:neuNS}. Then, for example by \cite[Theorem 2.6]{Bil1999}, the whole sequence $\L_{n}$ converges weakly to $\delta_\omega$, and then also in probability as the limit point is a Dirac delta (see e.g. the argument in \cite[page 27]{Bil1999}). The proof is complete.
\end{proof}

\subsection{Proof of \autoref{lem:tightness}}
Thanks to the uniform estimates of \autoref{sec:uniformbounds},
we can straightforwardly apply Aldous' criterion, \autoref{prop:aldous}, to the processes $\omega^n_t$ on $\DD([0,T];\HH^{\eta,p})$.

Condition (1), that is tightness of laws at fixed time in $\HH^{\eta,p}$, 
actually holds for \emph{all} $t\in [0,1]$.
This is a trivial consequence of \autoref{prop:unifspacebound},
Markov's inequality and the fact that $\HH^{\eta',p}\hookrightarrow \HH^{\eta,p}$
with $\eta'>\eta$ is a compact embedding.

Condition (2) of \autoref{prop:aldous} is the harder one.
Let us show that \eqref{eq:aldouscondition} holds: given a sequence $(\tau_n)_n$ of
$\F$-stopping times we estimate, for $q>2$,
\begin{multline*}
    \prob{\norm{\omega^n_{(\tau_n+r)\wedge 1}-\omega^n_{\tau_n}}_{\HH^{\eta,p}}>\delta}
    \leq \delta^{-q} \expt{\norm{\omega^n_{(\tau_n+r)\wedge 1}-\omega^n_{\tau_n}}_{\HH^{\eta,p}}^q}\\
    \lesssim_{\alpha,q} \delta^{-q} \expt{\norm{\omega^n_{(\tau_n+r)\wedge 1}-\omega^n_{\tau_n}}_{\HH^{-2,p}}^{\frac{q(\alpha-\eta)}{2+\alpha}}\cdot \sup_{t\in [0,1]}\norm{\omega^n_t}^{\frac{q(2+\eta)}{2+\alpha}}_{\HH^{\alpha,p}}}\\
    \lesssim_{\alpha,q} \delta^{-q} \expt{\norm{\omega^n_{(\tau_n+r)\wedge 1}-\omega^n_{\tau_n}}_{\HH^{-2,p}}^{\frac{uq(\alpha-\eta)}{2+\alpha}}} 
    \expt{\sup_{t\in [0,1]}\norm{\omega^n_t}^{\frac{u'q(2+\eta)}{2+\alpha}}_{\HH^{\alpha,p}}},
\end{multline*}
where the second step is the interpolation inequality between $\HH^{-2,p}$
and $\HH^{\alpha,p}$, and the third is H\"older inequality with exponents
$1<u,u'<\infty$. 
A simple computation reveals that, since $q>2$, one can choose $u$ so that the exponents
of the norms in the expected values are strictly larger than $2$, therefore
we can now apply the uniform estimates of
\autoref{prop:unifspacebound} and \autoref{prop:uniftimebound},
obtaining
\begin{equation*}
    \prob{\norm{\omega^n_{(\tau_n+r)\wedge 1}-\omega^n_{\tau_n}}_{\HH^{\eta,p}}>\delta}
    \lesssim_{\nu,M,\alpha,p,q} \delta^{-q} 
    \pa{r^{\frac{uq(\alpha-\eta)}{2(2+\alpha)}}+n^{-\frac{uq(\alpha-\eta)}{2+\alpha}}}.
\end{equation*}
It is now possible, given $\delta>0$, to choose $r_0>0$ small enough
and $n_0$ large enough (taking then $r\leq r_0$ and $n\geq n_0$) so that the right-hand side
of the latter inequality is arbitrarily small, thus satisfying Aldous' condition
and implying tightness of $(\L_n)_{n\in\N}$.

\subsection{Proof of \autoref{lem:limitingdynamics}}\label{sec:limit}
We consider a weak limit point $\L$ in $\DD([0,T];\HH^{\eta,p})$;
for simplicity, we still denote by $\L_n$ the convergent subsequence.

Consider the joint law of $\omega^n$ and the sequence of Brownian motions
$(B^i)_{i\in\N_0}$ on the product space $(\R^2)^\N\times \DD([0,T];\HH^{\eta,p})$
(which we regard as a product metric space, considering a distance on
$(\R^2)^\N$ that makes it separable).
In particular, the relation between Brownian motions and $\omega^n$
is encoded in the weak formulation of the dynamics:
\begin{align} \label{eq:weakformulation}
        \langle\omega^n_t,\phi\rangle-\langle\omega^n_0,\phi\rangle 
            &= \sum_{i\in A^n_0(t)}\omega_i^n \langle p_\epsilon(\cdot,\zeta_i^n),\phi\rangle
            +\nu \int_0^t \langle \omega^n_s,\Delta\phi\rangle ds\\ \nonumber
			&\quad +\int_0^t \int_D F\pa{K[\omega^n_s](y)}\cdot\langle\nabla_y p_\epsilon(\cdot,y),\phi\rangle dS^n_s(y)ds\\ \nonumber
			&\quad  +\sqrt{2\nu} \int_{0}^t \sum_{i\in A^n(s)}\omega_i^n \langle\nabla_y p_\epsilon(\cdot,x_i^n(s)),\phi\rangle \cdot dB^i_s.
\end{align}
$\PP$-a.s. for each $t\in [0,1]$ and $\phi\in \HH^{s+2,p'}$ with $sp'>2$ (and $ 1<p'< 2 $ since $p>2$).

By Skorohod coupling theorem on $(\R^2)^\N\times \DD([0,T];\HH^{\eta,p})$, 
\cite[Theorem 3.30]{Kallenberg2002},
there exists a complete probability space $(\tilde\Omega,\mathcal{A},\tilde\PP)$ with random elements
$(\tilde\omega_n)_{n\in\N}$, $(\tilde B^{i,n})_{i\in \N}$ having the same
joint laws of $(\omega_n)_{n\in\N}$, $(B^i)_{i\in \N}$,
and $\tilde \omega$ with law $\tilde \L=\L$, such that
$\tilde \omega_n$ converges $\tilde\PP$-almost surely to $\tilde\omega$. Moreover we can take the filtration $\tilde\F^{n,0}=(\tilde\F^{n,0}_t)_t$ generated by $(\tilde\omega^n,(\tilde{B}^{i,n})_i)$ and the $\PP$-null sets on $(\Omega,\mathcal{A})$ and we define the filtration $\tilde{\F}^n$ by $\tilde{\F}^n_t=\cap_{s>t}\tilde{\F}^{n,0}_t$; by a standard argument, see e.g. the proof of \cite[Proposition 2.5, part 1]{Bas2011}, we get that ($\tilde\omega^n$ and $(\tilde{B}^{i,n})_i$ are progressively measurable with respect to $\tilde{\F}^n$ and) $(\tilde{B}^{i,n})_i$ is still a cylindrical Brownian motion with respect to $\tilde{\F}^n$. Arguing for example as in \cite[Section 4.3.4]{Ben1995} (see also \cite[Proposition 4.1]{DebGlaTem2011} or \cite[Section 2.4.4]{flaWas}),
we can show that \eqref{eq:weakformulation} is satisfied by $\tilde \omega_n,(\tilde B^{i,n})_{i=1,\dots,N}$ on the filtered probability space $(\tilde{\Omega},\tilde{\F}^n,\PP)$.

If we now show that $\tilde \L=\delta_\omega$, so that also $\L=\delta_\omega$,
we obtain the first statement of the Lemma. In proving such claim,
we drop tilde signs from $\tilde{\omega}^n,(\tilde{B}^{i,n})_i$ to lighten notation\footnote{The tilde on our limit point $\tilde{\omega}$ is retained since a priori it can differ from the solution $\omega$ of \eqref{eq:neuNS}.}.
The aim is to pass to the limit \eqref{eq:weakformulation},
we do so term by term.

\subsubsection{Linear Terms} It is clear that, $\PP$-almost surely,
\begin{equation*}
    \brak{\omega^n_t,\phi}\to\brak{\tilde\omega_t,\phi},\quad 
    \nu\int_0^t\langle \omega^n_s,\Delta \phi\rangle ds\rightarrow \nu\int_0^t\langle \tilde \omega_s,\Delta \phi\rangle ds. 
\end{equation*}
We then need to prove that 
$\langle \omega_0^n,\phi\rangle\to \langle \omega_0,\phi  \rangle $,
that is, $\tilde\omega_0=\omega_0$; we do so
by showing that as $n\to\infty$, the following quantities vanish:
\begin{align*}
    I_1&:=\left\lvert \langle \omega_0^n,\phi\rangle-\int_D dx\int_D dy p_{\epsilon}(x,y)\omega_0(y)\phi(x) \right\rvert,\\ 
    I_2&:=\left\lvert \langle \omega_0,\phi  \rangle-\int_D dx\int_D dy p_{\epsilon}(x,y)\omega_0(y)\phi(x) \right\rvert.
\end{align*}
We treat $I_1$ exploiting the fact that $\phi$ is Lipschitz continuous,
\begin{multline*}
    I_1
    =\abs{\sum_{i:t_i=0}\omega_i^n\langle p_{\epsilon}(\cdot,\zeta_i^n),\phi\rangle-\int_D dx\int_D dy p_{\epsilon}(x,y)\omega_0(y)\phi(x)}\\
    =\abs{\sum_{i:t_i=0}\omega_i^n P_{\epsilon}\phi(\zeta_i^n)-\langle \omega_0,P_{\epsilon}\phi\rangle}
    \leq \sum_{i:t_i=0}\int_{Q_i^n}  \omega_0(y)\lvert P_{\epsilon}\phi(\zeta_i^n)-P_{\epsilon}\phi(y) \rvert dy\\
    \leq \sum_{i:t_i=0}\int_{Q_i^n}  \omega_0(y)\lvert y-\zeta_i^n\rvert \lVert P_{\epsilon}\phi \rVert_{C^1(\bar{D})} dy
    \lesssim \frac1n \norm{\phi}_{\HH^{s+2,p'}}\downarrow 0.
\end{multline*}
As for $I_2$, using the regularity of $\omega_0$ we assumed by hypothesis,
\begin{equation*}
    I_2=\left\lvert \langle \omega_0,\phi  \rangle- \langle P_{\epsilon}\omega_0,\phi \rangle\right\rvert \leq \lVert \phi\rVert_{L^2}\lVert (I-P_{\epsilon})\omega_0\rVert_{L^2}\downarrow 0.
\end{equation*}

\subsubsection{Boundary Generation Terms}
They converge to the PDE Neumann source,
conveniently represented by $\int_0^t\langle g_s,\phi\rangle_{\partial D} ds$.
The argument is analogous to parts of the proofs of
\autoref{lem:boundarygeneration} and \autoref{prop:uniftimebound}, 
therefore we omit some details. We split:
\begin{multline*}
    \left\lvert\sum_{i\in A^n_0(t)}\omega_i^n \langle p_\epsilon(\cdot,\zeta_i^n),\phi\rangle-\int_0^t\langle g_s,\phi\rangle_{\partial D} ds\right\rvert\\
    \leq \left\lvert\sum_{i\in A^n_0(t)}\omega_i^n \langle p_\epsilon(\cdot,\zeta_i^n),\phi\rangle-\int_0^t\int_{\partial D}g_s(y)\langle p_{\epsilon}(\cdot,y),\phi\rangle dy ds\right\rvert\\ 
    + \left\lvert\int_0^t\langle g_s,\phi\rangle_{\partial D} ds-\int_0^t\int_{\partial D}g_s(y)\langle p_{\epsilon}(\cdot,y),\phi\rangle dy ds\right\rvert =J_1+J_2.
\end{multline*}
Let us define $I^n(t)=(t_j^n-1/(2n),t_j^n+1/2n]$ with $t_j^n$ such that $t\in I^n(t)$.
We control $J_1$ exploiting again the fact that $\phi$ is Lipschitz,
\begin{align*}
    J_1&=\abs{\sum_{i\in A^n_0(t)}\omega_i^n P_{\epsilon}\phi(\zeta_i^n)-\int_0^t\int_{\partial D}g_s(y)P_{\epsilon}\phi(y) dy ds}\\ 
    &\leq \sum_{i\in A^n_0(t)}\int_{Q_i^n} g_s(y)\lvert P_{\epsilon}\phi(\zeta_i^n)-P_{\epsilon}\phi(y) \rvert ds dy
    +\int_{I^n(t)}\int_{\partial D}g_s(y)\lvert P_{\epsilon}\phi(y)\rvert dy ds\\ 
    &\lesssim \sum_{i\in A^n_0(t)}\int_{Q_i^n}g_s(y)\lvert \zeta_{i}^n-y\rvert\lVert P_{\epsilon}\phi\rVert_{C^1(\bar{D})}ds dy
    +\frac{1}{n}\norm{g}_{L^{\infty}([0,1]\times \partial D)}\norm{P_\epsilon\phi}_{C(\bar{D})}\\ 
    &\lesssim_{s,p} \frac{1}{n}\lVert \phi\rVert_{\HH^{s+2,p'}}\downarrow 0.
\end{align*}
Once again, the second term $J_2$ is the easier one:
\begin{multline*}
    J_2 \leq \|g\|_{L^1([0,1];L^\infty(\partial D))}\norm{(I-P_\epsilon)\phi}_{C(\bar{D})}\\
    \leq \|g\|_{L^1([0,1];L^\infty(\partial D))}\norm{(I-P_\epsilon)(I+A_{p'})\phi}_{L^{p'}(D)}\downarrow 0.
\end{multline*}

\subsubsection{Martingale Term} By It\^o isometry, 
\begin{multline*}
    \expt{\abs{\int_{0}^t \sum_{i\in A^n(s)}\omega_i^n \langle\nabla_y p_\epsilon(\cdot,x_i^n(s)),\phi\rangle \cdot dB^{i,n}_s}^2}\\
    =\expt{\sum_{i\in A^n(t)}\left(\omega^{n}_i\right)^2\int_{t_i^n}^t \lvert \nabla P_{\epsilon}\phi(x_i^n(s))\rvert^2 ds}\\
    \lesssim \frac{1}{n^2}\sup_{x\in\overline{ D}}\lvert \nabla P_{\epsilon}\phi(x)\rvert 
    \lesssim_{s,p} \frac{1}{n^2}\lVert \phi\rVert_{\HH^{s+2,p'}}\downarrow 0.
\end{multline*}

\subsubsection{Nonlinear Term}
Observe first that, since we are assuming that $\omega^n\to\tilde \omega$, $\PP$-almost surely in $\DD([0,1];\HH^{\eta,p})$ for all $\frac{2}{p}<\eta$, by \autoref{lem:biotsavart}
we also have that $K[\omega^n]\to K[\tilde \omega]$, $\PP$-almost surely
in $\DD([0,1];C^1(\bar{D};\R^2))$.
We now proceed to show that
\begin{multline}\label{eq:nlpasstolimit}
    L\coloneqq \left\lvert \int_0^t \int_D F\pa{K[\omega^n_s](y)}\nabla P_{\epsilon}\phi(y) dS^n_s(y)ds \right.\\
    \left.-\int_0^t \int_D F(K[\tilde \omega_s](y))\nabla \phi (y) \tilde \omega_s(y) dyds
    \right\rvert \xrightarrow{\PP-a.s.}0.
\end{multline}
To do so, we add and subtract the same quantity three times,
apply the triangular inequality and estimate differences.
For the sake of shorter formulas and clearer passages we will write
\begin{gather*}
    \mathbf{K}^n_t(x)= F(K[\omega^n_t](x)), \quad \mathbf{K}_t(x)=F(K[\tilde \omega_t](x)), \quad x\in \bar D;\\
    T_1=\int_0^t \int_D \mathbf{K}^n_t(y)\nabla P_{\epsilon}\phi(y) dS^n_s(y)ds, \quad
    T_2= \int_0^t \int_D \mathbf{K}^n_t(y)\nabla P_{\epsilon}\phi(y) \omega^n_s(y)dyds,\\
    T_3=\int_0^t \int_D \mathbf{K}^n_t(y)\nabla P_{\epsilon}\phi(y) \tilde \omega_s(y)dyds,\quad
    T_4=\int_0^t \int_D \mathbf{K}_t(y)\nabla P_{\epsilon}\phi(y) \tilde\omega_s(y)dyds,\\
    T_5=\int_0^t \int_D \mathbf{K}_t(y)\nabla \phi(y) \tilde \omega_s(y)dyds.
\end{gather*}
Notice that $T_1$ and $T_5$ coincide with the terms of the difference
in \eqref{eq:nlpasstolimit}, therefore the latter vanishes if
$|T_1-T_2|,|T_2-T_3|,|T_3-T_4|,|T_4-T_5|$ do so.
The first difference is the harder one, 
since it involves the pointwise difference between
the empirical measure and its regularization:
\begin{align*}
    |T_1-T_2|
    &=\left\lvert \int_0^t \sum_{i\in A^n(s)} \omega^n_i \mathbf{K}^n_t(x_i^n(s))\nabla P_{\epsilon}\phi(x_i^n(s)) ds\right.\\ 
    &\qquad\left.- \int_0^t \sum_{i\in A^n(s)}\omega^n_i\int_D \mathbf{K}^n_t(z)\nabla P_{\epsilon}\phi(z) p_{\epsilon}(z,x_i^n(s))dzds\right\rvert\\ 
    &\leq (\norm{\omega_0}_{L^1(D)}+\norm{g}_{L^1([0,1]\times\partial D)})\cdot\\
    &\quad \cdot \sup_{\substack{t\in [0,1],\\ y\in \bar{D}}}
    \int_D \abs{\mathbf{K}^n_t(y)\nabla P_{\epsilon}\phi(y)-\mathbf{K}^n_t(z)\nabla P_{\epsilon}\phi(z)} p_{\epsilon}(z,y) dz,\\
    & \lesssim \sup_{\substack{t\in [0,1],\\ y\in \bar{D}}}
    \int_D \left\lvert \mathbf{K}^n_t(y)\nabla P_{\epsilon}\phi(y)-\mathbf{K}^n_t(z)\nabla P_{\epsilon}\phi(y)\right\rvert p_{\epsilon}(z,y) dz\\ 
    &\qquad+\sup_{\substack{t\in [0,1],\\ y\in \bar{D}}}
    \int_D \left\lvert \mathbf{K}^n_t(z)\nabla P_{\epsilon}\phi(y)-\mathbf{K}^n_t(z)\nabla P_{\epsilon}\phi(z)\right\rvert p_{\epsilon}(z,y) dz.
\end{align*}
From there, since $F$ is Lipschitz continuous, and applying the pointwise estimate on heat kernel \eqref{eq:kernelestimatePW},
\begin{align*}
    |T_1-T_2|
    & \lesssim  \sup_{\substack{t\in [0,1],\\ y\in \bar{D}}}
    \int_D \norm{\nabla P_{\epsilon}\phi}_{C(\bar{D})}\lvert y-z\rvert \norm{K[\omega^n_t]}_{C^{1}(\bar{D})}\frac{1}{\epsilon}e^{\lvert z-y\rvert^2/(c\epsilon)} dz\\ 
    & \qquad +\sup_{\substack{t\in [0,1],\\ y\in \bar{D}}}
    \int_D \norm{\nabla P_{\epsilon}\phi}_{C^1(\bar{D})}\lvert y-z\rvert \norm{K[\omega^n_t]}_{C(\bar{D})}\frac{1}{\epsilon}e^{\lvert z-y\rvert^2/(c\epsilon)} dz\\ 
    & \lesssim_{s,p} \sup_{\substack{t\in [0,1],\\ y\in \bar{D}}}
    \int_D \norm{\phi}_{\HH^{s+2,p'}}\lvert y-z\rvert \norm{K[\omega^n_t]}_{C^{1}(\bar{D})}\frac{1}{\epsilon}e^{\lvert z-y\rvert^2/(c\epsilon)} dz\\ 
    & \leq \norm{\phi}_{\HH^{s+2,p'}}\sup_{y\in \bar{D},n\in\N}\norm{\omega^n}_{\mathbb{D}([0,1];\HH^{\eta,p})} \int_D \lvert y-z\rvert \frac{1}{\epsilon}e^{\lvert z-y\rvert^2/(c\epsilon)} dz
    \xrightarrow{\PP-a.s.}0.
\end{align*}
The other differences now are controlled by the almost sure convergence assumption
and functional analytic estimates already repeatedly applied. Bounding
\begin{align*}
    |T_2-T_3|
    & \leq M\int_0^t \int_D \norm{\nabla P_{\epsilon}\phi}_{C(\bar{D})}\lvert \omega^n_s(y)-\tilde \omega_s(y)\rvert dyds \\ 
    & \lesssim_{s,p,M} \norm{\phi}_{\HH^{s+2,p'}}\norm{\omega^n-\tilde\omega}_{\DD([0,1];\HH^{\eta,p})},\\
    |T_3-T_4| 
    &\lesssim_M \int_0^t \int_D \norm{\nabla P_{\epsilon}\phi}_{C(\bar{D})}\norm{\tilde \omega_s}_{C(\bar{D})}\lvert K[\omega_s^n](y)-K[\tilde \omega_s](y)\rvert dyds\\ 
    & \lesssim_{\eta,p}
     \norm{\nabla P_{\epsilon}\phi}_{C(\bar{D})}\norm{\tilde \omega}_{\DD([0,1];\HH^{\eta,p})}\lVert K\rVert_{H^{\eta,p}\rightarrow H^{\eta+1,p}}\norm{\omega^n-\tilde \omega}_{\DD([0,1];\HH^{\eta,p})},\\ 
    |T_4-T_5|
    & \le M\norm{\tilde \omega}_{\DD([0,1];C(\bar{D}))}\norm{\nabla(I-P_{\epsilon})\phi}_{C(\bar{D})}\\ 
    & \lesssim_{\eta,p,M}
     \norm{\tilde \omega}_{\DD([0,1];\HH^{\eta,p})}\norm{(I-P_{\epsilon})\phi}_{\HH^{s+2,p'}}\\ 
    & =M\norm{\tilde \omega}_{\DD([0,1];\HH^{\eta,p})}\norm{(I-P_{\epsilon})(I+A_{p'})^{(s+2)/2}\phi}_{L^{p'}},
\end{align*}
it is clear that differences on the right-hand side converge $\PP$-almost
surely to $0$ as $n\to\infty$.

\subsubsection{Removing the cutoff}
The argument detailed so far implies that the limit
$\tilde \omega$ is a weak solution, thus \emph{the unique} weak solution
of the cutoff PDE \eqref{eq:cneuNS}.
However, by \autoref{prop:pdeapriori} (second statement) it is now possible
to choose $M$ large enough so that $\norm{K[\tilde \omega]}_{L^\infty([0,1], L^\infty)}<M$, so that $F(K[\tilde \omega])=K[\tilde \omega]$ and thus
$\tilde \omega=\omega$ is actually the unique weak solution of the PDE \eqref{eq:neuNS}
without cutoff.
This concludes the proof of the first part of \autoref{lem:limitingdynamics}.

\section{General Initial and Boundary data}\label{sec:generalid}
In this section we provide a brief description of changes to the above arguments
required to deal with general boundary data $\omega_0\in \lip(\bar D) $ and $g\in \B_b([0,1]\times \de D)$, dropping the non-negativity assumption. 
The idea is simply to divide boundary data into positive and negative parts,
and to consider particle and PDE dynamics for the two parts,
properly taking into account their interactions.

\subsection{Splitting of PDE and particle dynamics}
Let us set
\begin{equation*}
    \omega_{0}^+:=\omega_0\lor 0,\quad \omega_{0}^-:=(-\omega_0)\lor 0,\quad g_{t}^+(x):=g_t(x)\lor 0,\quad g_{t}^-(x):=(-g_t(x))\lor 0. 
\end{equation*}
The limit dynamics \eqref{eq:neuNS} is then equivalent to a coupled system of PDEs,
\begin{equation}\label{eq:coupledsystem}
            \begin{cases}
            \de_t \omega^++K[\omega^+-\omega^-]\cdot \nabla \omega^+ =\nu \Delta \omega^+,\quad \text{in }D\times [0,T],\\
            \de_t \omega^-+K[\omega^+-\omega^-]\cdot \nabla \omega^- =\nu \Delta \omega^-,\quad \text{in }D\times [0,T],\\
            \nabla\omega^+\cdot \hat n=g^+,\quad \nabla\omega^-\cdot \hat n=g^-, \quad \quad \text{in }\de D\times (0,T), \\
            \omega^+(0)=\omega_0^+,\quad \omega^-(0)=\omega_0^-,\quad \text{in }D.
        \end{cases}
\end{equation}
The notion of weak solution to the latter system is completely analogous to the one of \autoref{def:weaksol}. Moreover, a couple $(\omega^+,\omega^-)$ is a weak solution of the system above if and only if the couple $(\omega=\omega^+-\omega^-,\omega^+)$ is a weak solution of 
\begin{equation}\label{eq:auxcoupledsystem}
            \begin{cases}
            \de_t \omega+K[\omega]\cdot \nabla \omega =\nu \Delta \omega,\quad \text{in }D\times [0,T],\\
            \de_t \omega^++K[\omega]\cdot \nabla \omega^+ =\nu \Delta \omega^+,\quad \text{in }D\times [0,T],\\
            \nabla\omega\cdot \hat n=g,\quad \nabla\omega^+\cdot \hat n=g^+, \quad \quad \text{in }\de D\times (0,T), \\
            \omega(0)=\omega_0,\quad \omega^+(0)=\omega_0^+,\quad \text{in }D.
        \end{cases}
\end{equation}
Well-posedness of \eqref{eq:auxcoupledsystem} follows easily from \autoref{prop:H-1method}. Therefore, also \eqref{eq:coupledsystem} is well posed in $\DD([0,T];L^2(D)^2)$.\\ 
As for the approximating dynamics, coherently with notation of \autoref{ssec:grid}, we define
\begin{equation}
    \omega_i^{n,\pm}=
    \begin{cases}
    \int_{Q_i^n}\omega_0^\pm(x)dx &t_i^n=0,\\
    \int_{Q_i^n} g_s^\pm(x) dsd\sigma(x) &t_i^n>0.
    \end{cases}
\end{equation}
From our assumptions on the boundary data and on the mesh of the grid,
it follows that $\omega_i^{n,\pm}\lesssim \frac{1}{n^2}$ uniformly in $i=1,\dots,N$,
for all $n$.
We introduce two mutually independent sequences of independent $\F$-Brownian motions $(B^{i,\pm})_{i\in\N}$ and particles with positions $x_i^{n,\pm}(t)$ satisfying the following evolution equations: for $t>t_i^n$,
    \begin{align*}
        x_i^{n,\pm}(t)-\zeta_i^n
        &=\int_{t_i^n}^t F \Bigg( 
        \sum_{j\in A^n(s)} \omega_j^{n,+} K_n(x_i^{n,\pm}(s),x_j^{n,+}(s))\\
        &\qquad -\omega_j^{n,-} K_n(x_i^{n,\pm}(s),x_j^{n,-}(s))
        \Bigg)ds
        + \sqrt{2\nu}\int_{t_i^n}^t dB^{i,\pm}_s- k_i^{n,\pm}(t),\\
        k_i^{n,\pm}(t)
        &=\int_{t_i^n}^t\hat n(x_i^{n,\pm}(s))d|k_i^{n,\pm}|(s), 
        \quad |k_i^{n,\pm}|(t)=\int_{t_i^n}^t\one_{x_i^{n,\pm}(s)\in \de D}d|k_i^{n,\pm}|(s).
    \end{align*}
where $k^{n,\pm}_i$ are continuous, adapted, $\R^2$-valued processes with bounded variation trajectories, whereas for $t\leq t_i^n$, $x_i^{n,\pm}(t)=\zeta_i^n$
and $k_i^{n,\pm}(t)=0$.
Well-posedness of this particle system follows from the same arguments proving \autoref{prop:wellposSDE}. Similarly to \autoref{sec:uniformbounds} we introduce the empirical measures 
\begin{equation*}
    S^{n,\pm}_t= \sum_{i\in A^n(t)}\omega_i^{n,\pm}\delta_{x_i^{n,\pm}(t)},
\end{equation*}
and the regularized empirical measures 
\begin{equation*}
    \omega^{n.\pm}_t(x)=P_\epsilon S^{n,\pm}_t= 
    \sum_{i\in A^n(t)}\omega_i^{n,\pm} p_\epsilon(x,x_i^{n,\pm}(t)),\quad x\in \bar D.
\end{equation*}
The following statement extends \autoref{thm:main} to the general case of
$\omega_0$, $g$ not necessarily non-negative.
\begin{theorem}\label{thm:mainsign}
Let $p>2,\ \frac{2}{p}<\alpha<1,\ \epsilon=\epsilon(n)\gtrsim n^{-1/2}$,
and $\omega_0\in \lip(\bar D) $, $g\in \B_b([0,1]\times \de D)$.
There exists $M>0$ (only depending on $\omega_0$, $g$) such that
for every $\eta\in (2/p,\alpha)$, as $n\to\infty$,
the stochastic process $(\omega^{n,+},\omega^{n,-})$ converges in probability on $\DD([0,T];\HH^{\eta,p}\times \HH^{\eta,p})$ to the unique weak solution of \eqref{eq:coupledsystem},
therefore $\omega^{n,+}-\omega^{n,-}$ converges in the same topology to the unique solution of \eqref{eq:neuNS}.
\end{theorem}

\subsection{Uniform Estimates for split dynamics}

The regularized empirical measures satisfy the integral equations and the mild formulations below:
\begin{align} \label{eq:itoOmegasign}
        \omega^{n,\pm}_t(x)-\omega^{n,\pm}_0(x) 
            &= \sum_{i\in A^n_0(t)}\omega_i^{n,\pm} p_\epsilon(x,\zeta_i^n)
            +\nu \int_0^t \Delta \omega^{n,\pm}_s(x)ds\\ \nonumber
			&\quad +\int_0^t \int_D F\pa{K[\omega^{n,+}_s-\omega^{n,-}_s](y)}\nabla_y p_\epsilon(x,y)dS^{n,\pm}_s(y)ds\\ \nonumber
			&\quad  +\sqrt{2\nu} \int_{0}^t \sum_{i\in A^n(s)}\omega_i^{n,\pm} \nabla_y p_\epsilon(x,x_i^{n,\pm}(s))dB^{i,\pm}_s.
\end{align}
\begin{align} \label{eq:duhamelformulasign}
        \omega^{n,\pm}_t(x)
            &=P_t\omega^{n,\pm}_0(x) +\sum_{i\in A^n_0(t)}\omega_i^{n,\pm} P_{t-t_i^n} p_\epsilon(x,\zeta_i^n)\\ \nonumber
			&\quad +\int_0^t P_{t-s} \int_D F\pa{K[\omega^{n,+}_s-\omega^{n,-}_s](y)}\nabla_y p_\epsilon(x,y)dS^{n,\pm}_s(y)ds\\ \nonumber
			&\quad +\sqrt{2\nu} \int_{0}^t \sum_{i\in A^n(s)}\omega_i^{n,\pm} P_{t-s} \nabla_y p_\epsilon(x,x_i^{n,\pm}(s))dB^{i,\pm}_s.   
\end{align}

In order to obtain the required compactness we need to show that 	\begin{equation}\label{eq:spaceboundcoupled}
		\sup_{n}\expt{\sup_{t\in [0,1]}\norm{\omega^{n,+}_t}^q_{\HH^{\alpha,p}}}+\sup_{n}\expt{\sup_{t\in [0,1]}\norm{\omega^{n,-}_t}^q_{\HH^{\alpha,p}}}<C_{M,q,p,\nu,\alpha},
	\end{equation}
 \begin{gather} \label{eq:timeboundcoupled}
		\expt{\norm{\omega^{n,+}_{\tau_n+r}-\omega^{n,+}_{\tau_n}}^q_{\HH^{-2,p}}}+\expt{\norm{\omega^{n,-}_{\tau_n+r}-\omega^{n,-}_{\tau_n}}^q_{\HH^{-2,p}}}
		\lesssim_{M,q,p,\nu,\alpha} \left(r^{q/2}+\frac{1}{n^q}\right).
	\end{gather}
The proof of these inequalities follows the same strategy described in \autoref{sec:uniformbounds}.
Indeed, linear terms, stochastic integrals and generation terms can be treated exactly as in the case of positive data. In fact, nonlinear terms present no additional difficulties and the crucial estimate for the nonlinear term of \autoref{prop:unifspacebound} is easily adapted. For $f\in L^{p'}$ such that $\norm{f}_{L^{p'}}=1$ the following chain of inequalities holds:
   \begin{align*} 
       &\Bigg\| \sum_{i\in A^n(s)}\omega_i^{n,\pm} F\left(K[\omega^{n,+}_s-\omega^{n,-}_s](x_i^{n,\pm}(s))\right)\\
       &\qquad\qquad \cdot\int_D f(x)(I+A_p)^{\alpha/2}P_{t-s}
       \nabla_y p_\epsilon(x,x_i^{n,\pm}(s))dx\Bigg\|\\
       &\quad \lesssim_M\sum_{i\in A^n(s)}\omega_i^{n,\pm}  \abs{\nabla P_\epsilon[(I+A_{p'})^{\alpha/2}P_{t-s}f](x_i^{n,\pm}(s))}\\
       &\quad \lesssim \sum_{i\in A^n(s)}\omega_i^{n,\pm} P_\epsilon\abs{\nabla(I+A_{p'})^{\alpha/2}P_{t-s}f}(x_i^{n,\pm}(s))\\
       &\quad = \sum_{i\in A^n(s)}\omega_i^{n,\pm}  \int_D p_\epsilon(x,x_i^{n,\pm}(s))  \abs{\nabla(I+A_{p'})^{\alpha/2}P_{t-s}f}(x)dx\\
       &\quad = \int_D P_\epsilon(S^{n,\pm})(x)
      \abs{\nabla(I+A_{p'})^{\alpha/2}P_{t-s}f}(x)dx\\ 
      &\quad \leq  \norm{\omega^{n,\pm}_s}_{L^p} \norm{\nabla(I+A_{p'})^{\alpha/2}P_{t-s}f}_{L^{p'}}
      \lesssim \frac{1}{(t-s)^{(1+\alpha)/2}} \norm{\omega^{n,\pm}_s}_{\HH^{\alpha,p}}.
   \end{align*}
   Then the proof of \eqref{eq:spaceboundcoupled} follows exactly as in \autoref{prop:unifspacebound}.
   Similarly, in the proof of \eqref{eq:timeboundcoupled} 
   the only differences concern the nonlinear terms, and they can be handled
   with analogous simple modifications.

   Once uniform bounds are recovered, one can proceed to replicate the tightness argument of \autoref{sec:compactness}. 
   Passing to the limit dynamics is similar to that explained in \autoref{sec:limit},
   once again we only outline the (slightly) different treatment required by nonlinear terms.
   
   We consider the weak formulation satisfied by $\omega^{n,\pm}_t$, $t\in [0,1]$ and test functions $(\phi^{+},\phi^{-})\in \HH^{s+2,p'}\times \HH^{s+2,p'}$ with $sp'>2,\ 1<p'< 2 $,
   \begin{align} \label{eq:weakformulationcoupled}
        &\brak{\omega^{n,\pm}_t,\phi^{\pm}\rangle-\langle\omega^{n,\pm}_0,\phi^{\pm}}\\ \nonumber
        &\qquad= \sum_{i\in A^n_0(t)}\omega_i^{n,\pm} \langle p_\epsilon(\cdot,\zeta_i^n),\phi^{\pm}\rangle
            +\nu \int_0^t \langle \omega^{n,\pm}_s,\Delta\phi^{\pm}\rangle ds\\ \nonumber
		&\qquad +\int_0^t \int_D F\pa{K[\omega^{n,+}_s-\omega^{n,-}_s](y)}\cdot\langle\nabla_y p_\epsilon(\cdot,y),\phi^{\pm}\rangle dS^{n,\pm}_s(y)ds\\ \nonumber
		&\qquad  +\sqrt{2\nu} \int_{0}^t \sum_{i\in A^n(s)}\omega_i^{n,\pm} \langle\nabla_y p_\epsilon(\cdot,x_i^{n,\pm}(s)),\phi^{\pm}\rangle \cdot dB^{i,\pm}_s.
\end{align}
We stress once again that linear terms, stochastic integrals and generation terms present no changes with respect to \autoref{sec:limit}.
Since (by means of a Skorohod argument) we can assume that 
$\omega^{n,\pm}\stackrel{\PP-a.s.}{\rightarrow}\omega^{\pm}$ in $\DD([0,1];\HH^{\eta,p}) $ for all $\frac{2}{p}<\eta$,
and by the properties of the Biot-Savart Kernel, 
\begin{equation*}
    u^n:=K[\omega^{n,+}-\omega^{n,-}]\stackrel{\PP-a.s.}{\rightarrow}K[\omega^+-\omega^-]=:u\quad  \text{ in }\DD([0,1];C^1(\bar{D};\R^2)).
\end{equation*}
A careful analysis of computations in \autoref{sec:limit} reveals that only the terms $\lvert T_1^{\pm}-T_2^{\pm}\rvert,$ and $\lvert T_3^{\pm}-T_4^{\pm}\rvert$ are affected by the generalization to the coupled positive-negative system. 
However, only little changes are needed:  
\begin{align*}
    \lvert T_1^{\pm}-T_2^{\pm}\rvert
    &=\bigg\lvert \int_0^t \sum_{i\in A^n(s)} \omega^{n,\pm}_iF\pa{u^n_s(x_i^{n,\pm}(s))} \nabla P_{\epsilon}\phi^{\pm}(x_i^{n,\pm}(s)) ds\\ 
    &\qquad - \int_0^t \sum_{i\in A^n(s)}\omega^{n,\pm}_i\int_D F\pa{u^n_s(z)} \nabla P_{\epsilon}\phi^{\pm}(z) p_{\epsilon}(z,x_i^{n,\pm}(s))dzds\bigg\rvert\\ 
    & \lesssim \sup_{\substack{t\in [0,1],\\ y\in \bar{D}}}
    \int_D \left\lvert F\pa{u^n_t(y)} \nabla P_{\epsilon}\phi^{\pm}(y)-F\pa{u^n_t(z)} \nabla P_{\epsilon}\phi^{\pm}(z)\right\rvert p_{\epsilon}(z,y) dz\\ 
    & \leq 
    \sup_{\substack{t\in [0,1],\\ y\in \bar{D}}} \int_D \left\lvert F\pa{u^n_t(y)} \nabla P_{\epsilon}\phi^{\pm}(y)-F\pa{u^n_t(z)} \nabla P_{\epsilon}\phi^{\pm}(y)\right\rvert p_{\epsilon}(z,y) dz\\ 
    & \qquad +\sup_{\substack{t\in [0,1],\\ y\in \bar{D}}} \int_D \left\lvert F\pa{u^n_t(z)} \nabla P_{\epsilon}\phi^{\pm}(y)-F\pa{u^n_t(z)} \nabla P_{\epsilon}\phi^{\pm}(z)\right\rvert p_{\epsilon}(z,y) dz\\ 
    & \lesssim  \sup_{\substack{t\in [0,1],\\ y\in \bar{D}}} \int_D \norm{\nabla P_{\epsilon}\phi^{\pm}}_{C(\bar{D})}\lvert y-z\rvert \norm{u^n_t}_{C^{1}(\bar{D})}\frac{1}{\epsilon}e^{\lvert z-y\rvert^2/(c\epsilon)} dz\\ 
    & \qquad +\sup_{\substack{t\in [0,1],\\ y\in \bar{D}}} \int_D \norm{\nabla P_{\epsilon}\phi^{\pm}}_{C^1(\bar{D})}\lvert y-z\rvert \norm{u^n_t}_{C(\bar{D})}\frac{1}{\epsilon}e^{\lvert z-y\rvert^2/(c\epsilon)} dz\\ 
    & \lesssim \sup_{\substack{t\in [0,1],\\ y\in \bar{D}}} \int_D \norm{\phi^{\pm}}_{\HH^{s+2,p'}}\lvert y-z\rvert \norm{u^n_t}_{C^{1}(\bar{D})}\frac{1}{\epsilon}e^{\lvert z-y\rvert^2/(c\epsilon)} dz\\ 
    & \lesssim \norm{\phi^{\pm}}_{\HH^{s+2,p'}}\sup_{y\in \bar{D},n\in\N}\norm{\omega^{n,+}-\omega^{n,-}}_{\DD([0,1];\HH^{\eta,p})}
    \xrightarrow{\PP-a.s.}0,\end{align*}\begin{align*}
    \lvert T_3^{\pm}-T_4^{\pm}\rvert 
    &\lesssim \int_0^t \int_D \norm{\nabla P_{\epsilon}\phi^{\pm}}_{C(\bar{D})}\norm{\omega^{\pm}_s}_{C(\bar{D})}\lvert u^n_s(y)-u_s(y)\rvert dyds\\ 
    & \lesssim
     \norm{\nabla P_{\epsilon}\phi^{\pm}}_{C(\bar{D})}\norm{\omega^{\pm}}_{\DD([0,1];\HH^{\eta,p})}\lVert K\rVert_{H^{\eta,p}\rightarrow H^{\eta+1,p}}\\ 
     &\quad\left(\norm{\omega^{n,+}-\omega^{+}}_{\DD([0,1];\HH^{\eta,p})}+\norm{\omega^{n,-}-\omega^{-}}_{\DD([0,1];\HH^{\eta,p})}\right)\xrightarrow{\PP-a.s.}0.
\end{align*}
The remaining passages then coincide with the non-negative case discussed above. 

\begin{acknowledgements}
    We thank Margherita Zanella for useful discussions. F.G. was supported by the project \emph{Mathematical methods for climate science} funded by PON R\&I 2014-2020 (FSE REACT-EU).
\end{acknowledgements}

\bibliographystyle{plain}

\end{document}